\documentclass[11pt, a4paper]{amsart}

\usepackage{amsthm,amsfonts,amsbsy,amssymb,amsmath,amscd}
\usepackage[all]{xy}
\usepackage[colorlinks=true]{hyperref}

\hypersetup{
	colorlinks,
	linkcolor=purple,
	citecolor=blue
}

\usepackage{pgf, tikz}

\newcommand{\Alb}{{\mathrm{Alb}}}
\newcommand{\Aut}{{\mathrm{Aut}}}
\newcommand{\Gal}{{\mathrm{Gal}}}
\newcommand{\Pic}{{\mathrm{Pic}}}

\newcommand{\vol}{{\mathrm{Vol}}}
\newcommand{\rank}{{\mathrm{rank}}}
\newcommand{\albdim}{{\mathrm{albdim}}}
\newcommand{\Spec}{{\mathrm{Spec}}}

\newcommand{\nti}{{\mathrm{int}}}

\newcommand{\CO}{{\mathcal{O}}}
\newcommand{\CE}{{\mathcal{E}}}

\newcommand{\CP}{{\mathcal{P}}}
\newcommand{\CQ}{{\mathcal{Q}}}

\newcommand{\PP}{{\mathbb{P}}}
\newcommand{\QQ}{{\mathbb{Q}}}
\newcommand{\ZZ}{{\mathbb{Z}}}

\newcommand{\rounddown}[1]{{\lfloor #1 \rfloor}}

\begin{document}

%Theorems

\theoremstyle{plain}
\newtheorem{theorem}{Theorem}[section]
\newtheorem{lemma}[theorem]{Lemma}
\newtheorem{coro}[theorem]{Corollary}
\newtheorem{prop}[theorem]{Proposition}
\newtheorem{defi}[theorem]{Definition}
\newtheorem{ques}[theorem]{Question}
\newtheorem{conj}[theorem]{Conjecture}

\newtheorem*{ques*}{Question}
\newtheorem*{ques'*}{Question'}
\newtheorem*{conj*}{Conjecture}

\theoremstyle{remark}
\newtheorem{remark}[theorem]{\bf Remark}
\newtheorem{assumption}[theorem]{\bf Assumption}
\newtheorem{example}[theorem]{\bf Example}

\numberwithin{equation}{section}

\setcounter{tocdepth}{1} 

%%%%%

\title[Relative Severi inequality for mAd fibrations]{Relative Severi inequality for fibrations of maximal Albanese dimension over curves}

\author{Yong Hu}
\author{Tong Zhang}
\date{\today}

\address[Yong Hu]{School of Mathematical Sciences, Shanghai Jiao Tong University, 800 Dongchuan Road, Shanghai 200240, People's Republic of China}
\email{yonghu@sjtu.edu.cn}

\address[Tong Zhang]{School of Mathematical Sciences, Shanghai Key Laboratory of PMMP, East China Normal University, 500 Dongchuan Road, Shanghai 200241, People's Republic of China}
\email{tzhang@math.ecnu.edu.cn, mathtzhang@gmail.com}


\begin{abstract}
	Let $f: X \to B$ be a relatively minimal fibration of maximal Albanese dimension from a variety $X$ of dimension $n \ge 2$ to a curve $B$ defined over an algebraically closed field of characteristic zero. We prove that $K_{X/B}^n \ge 2n! \chi_f$. It verifies a conjectural formulation of Barja in \cite{Barja_Severi}. Via the strategy outlined in \cite{Barja_Stoppino}, it also leads to a new proof of the Severi inequality for varieties of maximal Albanese dimension. Moreover, when the equality holds and $\chi_f > 0$, we prove that the general fiber $F$ of $f$ has to satisfy the Severi equality that $K_F^{n-1} = 2(n-1)! \chi(F, \omega_F)$. We also prove some sharper results of the same type under extra assumptions.
\end{abstract}

\maketitle

\tableofcontents

\section{Introduction}

The Severi inequality states that
$$
K_X^n \ge 2n! \chi(X, \omega_X)
$$
for an $n$-dimensional minimal variety $X$ of general type and of maximal Albanese dimension. It was originally stated for surfaces by Severi \cite{Severi} and was proved by Pardini \cite{Pardini}. Later, it was generalized to arbitrary dimension by Barja \cite{Barja_Severi} as well as the second author \cite{Zhang_Severi}. From now on, we refer this inequality as the \emph{absolute} Severi inequality in order to distinguish from the result in the current paper.

The goal of this paper is to establish a \emph{relative} version of the absolute Severi inequality. More precisely, we prove that
$$
K_{X/B}^n \ge 2n! \chi_f
$$
for a relatively minimal fibration $f: X \to B$ of maximal Albanese dimension from an $n$-dimensional variety $X$ to a curve $B$. This inequality was conjecturally formulated by Barja in \cite[\S 1]{Barja_Severi}. The $g(B) = 0$ case of this relative inequality can be applied to give a new proof of the above absolute Severi inequality. Moreover, the above relative inequality is sharp, and if $K_{X/B}^n = 2n! \chi_f > 0$, we prove that the general fiber $F$ of $f$ has to satisfy the absolute Severi equality that
$$
K_F^{n-1} = 2(n-1)! \chi(F, \omega_F).
$$

We also use our method to deduce some shaper relative results of the same type under extra assumptions. As an upshot, the corresponding $g(B) = 0$ case implies the recent geographical results of absolute Severi type obtained by Barja, Pardini and Stoppino \cite{Barja_Pardini_Stoppino}.

Throughout this paper, we work over an arbitrary algebraically closed field $k$ of characteristic zero. All varieties are assumed to be projective.

\subsection{Albanese dimension of fibrations and $\chi_f$} \label{subsection: Albanese dimension}
We start from some notation. In the study of irregular varieties, a major tool is to consider the Albanese map. For an irregular variety $X$, the so-called Albanese dimension $\albdim(X)$ of $X$ is one of the most important invariants of $X$. In the following, we consider its relative version.

Let $f: X \to Y$ be a fibration between two normal varieties $X$ and $Y$ with a general fiber $F$. Let $a: X \to \Alb(X)$ be the Albanese map of $X$. 

\begin{defi} \label{def: Albanese dimension}
	The \emph{Albanese dimension of $f$}, denoted by $\albdim(f)$, is defined to be $\dim a(F)$, namely the dimension of the image of $F$ under the Albanese map of $X$. We say that $f$ is \emph{of maximal Albanese dimension}, if $\albdim(f) = \dim F$.
\end{defi}

It is easy to check that the following properties hold:

\begin{itemize}
	\item [(1)] When $f$ is the structural morphism, i.e., $Y = \Spec (k)$, then 
	$$
	\albdim(f) = \albdim(X).
	$$
	Thus the Albanese dimension for fibrations is indeed a generalization of that for varieties.
	\item [(2)] In general, we have
	$$
	\albdim(f) \le \albdim(X) - \albdim(Y).
	$$
	In particular, if $f$ is the Stein factorization of the Albanese map of $X$, then $\albdim(f) = 0$.
	\item [(3)] If both $Y$ and $f$ are of maximal Albanese dimension, so is $X$.
\end{itemize}

Another important invariant associated to $f$ is the relative Euler characteristic
$$
\chi_f := \chi(X, \omega_X) - \chi(Y, \omega_Y) \chi(F, \omega_F).
$$
Regarding this invariant, the first interesting case is when $f: X \to Y$ is a surface fibration, i.e., $X$ is a smooth surface and $Y$ is a curve. In this case, it is well-known that
$$
\chi_f = \deg f_* \omega_{X/Y}.
$$ 
In particular, by \cite[Main Theorem]{Fujita}, we know that $\chi_f \ge 0$. There are a number of important results related to $\chi_f$, such as the Arakelov inequality \cite{Arakelov} (see \cite{Viehweg} for a survey together with generalizations), the slope inequality of Cornalba-Harris \cite{Cornalba_Harris} and Xiao \cite{Xiao}, the geography of irregular surfaces (see \cite{Lopes_Pardini} for a detailed survey). The study of these results as well as their refinements and generalizations has always been active throughout the past decades.

Another interesting case, which is more related to this paper, is when $f$ is a fibration of maximal Albanese dimension and $Y$ is a curve. In this case, by the work of Hacon and Pardini \cite[Theorem 2.4]{Hacon_Pardini} (see Proposition \ref{prop: chi_f degree} for a slightly generalized version adapting to the setting of this paper), we know that 
$$
\chi_f = \deg f_* (\omega_{X/Y} \otimes \CP),
$$
where $\CP$ is a general torsion element in $\Pic^0(X)$. Moreover, they showed loc. cit. that $\chi_f \ge 0$ still holds in this case.

\subsection{Main results}
Now we state the first main theorem of this paper.

\begin{theorem} [Relative Severi inequality] \label{main1}
	Let $f: X \to B$ be a relatively minimal fibration from a variety $X$ of dimension $n \ge 2$ to a smooth curve $B$. Suppose that $f$ is of maximal Albanese dimension. Then we have the following sharp inequality
	\begin{equation} \label{main inequality}
		K_{X/B}^n \ge 2n! \chi_f.
	\end{equation}
\end{theorem}

We call the inequality \eqref{main inequality} a \emph{relative Severi inequality} because it literally replaces the absolute invariants $K_X^n$ and $\chi(X, \omega_X)$ in the absolute Severi inequality by the relative invariants $K_{X/B}^n$ and $\chi_f$. 

Let us put Theorem \ref{main1} into perspective. When $n=2$, it has already been known by Xiao \cite[Corollary 1]{Xiao}. More precisely, Xiao proved that for a relatively minimal surface fibration $f: X \to B$ with a general fiber of genus $g \ge 2$, the inequality \eqref{main inequality} holds provided that $h^1(X, \CO_X) > g(B)$. Note that this assumption is equivalent to that $f$ is of maximal Albanese dimension, as the fiber in this case is just a curve. 

For general $n > 2$, the problem about finding such kind of inequalities has already been addressed by Mendes Lopes and Pardini \cite[\S 5.3]{Lopes_Pardini}, whose purpose was to generalize, using Pardini's original approach in \cite{Pardini}, the Severi inequality for surfaces to higher dimensions. To our knowledge, the precise version of \eqref{main inequality} was first formulated conjecturally by Barja in \cite[\S 1, Page 545]{Barja_Severi}. Barja also observed loc. cit. that \eqref{main inequality} is in fact a consequence of the $f$-positivity conjecture \cite[Conjecture 1]{Barja_Stoppino} due to himself and Stoppino.\footnote{This conjecture was recently studied by the authors in \cite{Hu_Zhang}, where it is shown that counterexamples to this conjecture do exist for any $n>2$.} Another interesting observation, which probably motivates the formulation \eqref{main inequality}, is that when $X$ itself is of maximal Albanese dimension, one can indeed deduce the absolute Severi inequality just combining Pardini's approach and \eqref{main inequality} for $g(B) = 0$ (see \cite[Proposition 4.4]{Barja_Stoppino} for details). 

When $g(B) = 1$, it is easy to see that \eqref{main inequality} coincides with the absolute Severi inequality. Besides this and prior to our result, Barja has proved \eqref{main inequality} for $g(B) = 0$ under extra assumptions that $X$ is of maximal Albanese dimension and that $K_X$ is nef. Barja also obtained a weaker version of \eqref{main inequality} when $g(B) \ge 2$. See \cite[Corollary C]{Barja_Severi} as well as its proof for details.

Our Theorem \ref{main1} verifies completely the conjectural formulation of Barja for the base curve $B$ of arbitrary genus. Moreover, if $g(B) = 0$, our assumption that $f$ is of maximal Albanese dimension is strictly weaker than $X$ itself being of maximal Albanese dimension. As is mentioned before, Theorem \ref{main1} for $g(B) = 0$ can be applied to give an alternative proof of the absolute Severi inequality which is different from those in \cite{Barja_Severi} or \cite{Zhang_Severi}.\footnote{Since a detailed strategy has been carried out in \cite[Proposition 4.4]{Barja_Stoppino}, we will not repeat this proof in this paper and just refer the reader to loc. cit. for details.}

Since \eqref{main inequality} is sharp, a new question naturally arises: can one characterize the equality case? In this paper, we also consider this problem. We prove the following result.

\begin{theorem} \label{main2}
	In Theorem \ref{main1}, if the equality in \eqref{main inequality} holds and $\chi_f > 0$, then 
	\begin{itemize}
		\item [(1)] the Albanese map of $X$ maps a general fiber of $f$ onto an abelian variety of dimension $n-1$. In particular, 
		$$
		h^1(X, \CO_X) - g(B) = n-1;
		$$
		\item [(2)] the general fiber $F$ of $f$ satisfies the absolute Severi equality, i.e., 
		$$
		K_F^{n-1}=2(n-1)!\chi(F, \omega_F).
		$$
	\end{itemize}
\end{theorem}

Previously, (1) was known only when $n=2$ due to Xiao \cite[Theorem 3]{Xiao}. This paper mainly concerns the higher dimensional case, and our result shows that (1) holds for any $n \ge 2$. The much more interesting and stronger part comes from (2): not like (1) or the absolute Severi inequality, (2) is trivial when $n=2$, i.e., when the fiber is a curve, which says that $\deg K_F = 2\chi(F, \omega_F)$. It actually holds true for any surface fibration, not necessary of maximal Albanese dimension. However, for $n > 2$, (2) was completely unknown before, and it reveals a new connection between the geometry of a family of higher dimensional varieties and the geometry of a general member in this family.

Recall that for a surface fibration $f: X \to B$, the relative irregularity is defined as $q_f := h^1(X, \CO_X) - g(B)$. Recently, Pardini proposed a problem \cite[Problem 2]{Problem} to study various notions of relative irregularity for families of higher dimensional varieties. The result (1) also sheds some light on this problem, suggesting that the number $h^1(X, \CO_X) - g(B)$ may also serve as the relative irregularity for higher dimension fibrations over curves.

When $\dim F \ge 2$, by a very recent result of Barja, Pardini and Stoppino \cite[Theorem 1.2]{Barja_Pardini_Stoppino_equality} characterizing the variety satisfying the absolute Severi equality (see also \cite{Barja_Pardini_Stoppino_Severi, Lu_Zuo} when $\dim F = 2$), we know that  (2) actually implies (1). However, our proof of (1) is independent of (2).

%It seems difficult to use the method in \cite{Barja_Pardini_Stoppino_equality} to deduce Theorem \ref{main2}. First, if $g(B) = 0$, then $h^1(X, \CO_X) = n - 1$ and $X$ is not of maximal Albanese dimension. Even when $X$ itself is of maximal Albanese dimension, if $g(B) > 1$, then $h^1(X, \CO_X) > n$. This means that the Albanese image of $X$ in this case is no longer an abelian variety. In particular, $K_{X/B}$ does not satisfy the Clifford-Severi equality in \cite{Barja_Pardini_Stoppino_equality}. 

\subsection{Related results}

If more assumptions on the Albanese map of $X$ are imposed, we obtain sharper results. For example, we prove the following theorem.

\begin{theorem} \label{main3}
	Let $f: X \to B$ be a relatively minimal fibration from a variety $X$ of dimension $n \ge 3$ to a smooth curve $B$. Denote by $F$ a general fiber of $f$. Suppose that $f$ is of maximal Albanese dimension and $a: X \to \Alb(X)$ is the Albanese map of $X$.
	\begin{itemize}
		\item [(1)] If $a|_F$ is birational, then 
		$$
		K_{X/B}^n \ge \frac{5n!}{2} \chi_f.
		$$
		
		\item [(2)] If $a|_F$ is not composed with an involution, then
		$$
		K_{X/B}^n \ge \frac{9n!}{4} \chi_f.
		$$
	\end{itemize}
\end{theorem}

Combining Theorem \ref{main3} in the $g(B) = 0$ case with the method in \cite[Proposition 14]{Barja_Stoppino}, it is easy to get the following conclusion which was recently obtained by Barja, Pardini and Stoppino in \cite[\S 1]{Barja_Pardini_Stoppino}.

\begin{coro} \label{main4}
	Let $X$ be a minimal variety of general type of dimension $n \ge 3$. Suppose that $X$ is of maximal Albanese dimension. 
	\begin{itemize}
		\item [(1)] If the Albanese map of $X$ is birational onto its image, then
		$$
		K_X^n \ge \frac{5n!}{2} \chi(X, \omega_X).
		$$
		
		\item [(2)] If the Albanese map of $X$ is not composed with an involution, then 
		$$
		K_X^n \ge \frac{9n!}{4} \chi(X, \omega_X).
		$$
	\end{itemize}
\end{coro}

In the same spirit as before, we may view Theorem \ref{main3} as a relative version of Corollary \ref{main4}.

In \cite{Barja_Pardini_Stoppino}, Barja, Pardini and Stoppino consider a more general map $a: X \to A$ such that $a^*: \Pic^0(A) \to \Pic^0(X)$ is injective (which they call strongly generating), and prove Corollary \ref{main4} when $a$ is birational or when $a$ is not composed with an involution. In fact, by the universal property of the Albanese map, we see that if the $a$ is birational or is not composed with an involution, so is the Albanese map of $X$.

Furthermore, we would like to mention that the proof of the absolute Severi type inequalities by Barja, Pardini and Stoppino in \cite{Barja_Pardini_Stoppino} relies on their study of the continuous rank function. More precisely, they deduce these absolute results by integrating the derivative of the so-called continuous rank function. From the viewpoint of our paper, those absolute inequalities are just consequences of their corresponding relative counterparts. To summarize, we have seen again, as in the work of Pardini \cite{Pardini}, that the study of the \emph{relative geography}, namely the relation among relative birational invariants (such as the relative canonical volume, the relative Euler characteristic, etc) does play a crucial role in understanding the geography of algebraic varieties in the classical sense.

\subsection*{Notation and conventions} In this paper, a fibration always means a surjective morphism with connected fibers.

Let $f: X \to B$ be a fibration over a curve $B$. We say that $f$ is \emph{relatively minimal}, if $X$ is normal with at worst terminal singularities and $K_X$ is $f$-nef. The assumption implies that a general fiber $F$ of $f$ is also normal with at worst terminal singularities by the adjunction. Moreover, if a general fiber of $f$ is of maximal Albanese dimension (which is exactly under the setting of Theorem \ref{main1}), then the relative minimality also ensures that $K_{X/B}$ is nef.\footnote{In fact, Fujino \cite[Theorem 1.1]{Fujino} proved that in this case, the general fiber has a good minimal model. Thus by a result of Nakayama \cite[Theorem 5]{Nakayama}, $K_X$ is $f$-semi-ample. Using the argument as in the proof of \cite[Theorem 1.4]{Ohno}, we deduce that $K_{X/B}$ is nef.}

For divisors, we always use $\sim$ to denote the linear equivalence and use $\equiv$ to denote the numerical equivalence. Let $D_1$ and $D_2$ be two $\QQ$-divisors on a variety $V$. The notation $D_1 \ge D_2$ means that $D_1-D_2$ is effective. Let $D$ be a $\QQ$-divisor on $V$. We use $\rounddown{D}$ to denote its integral part. The volume of $D$ is defined as
$$
\vol(D):=\limsup_{m \to \infty} \frac{h^0(V, \rounddown{mD})}{m^{\dim V} / (\dim V)!}.
$$

\subsection*{Acknowledgment} 
Y.H. would like to thank Professors JongHae Keum and Jun-Muk Hwang for their generous support during  his stay at KIAS. T.Z. would like to thank Professor Miguel \'A. Barja for the comment on his conjectural inequality \eqref{main inequality} in an email in 2017 and a lot more valuable comments on the first version of this paper, as well as Professor Kang Zuo for many enlightening comments about Viehweg's result in \cite{Viehweg} which is crucial for proving Theorem \ref{main2}. T.Z also would like to thank Professors Zhi Jiang and Lidia Stoppino for their interest in this paper. Both authors would like to thank the anonymous referee sincerely for his/her comments and suggestions. 

Y. H. is supported by National Researcher Program of National Research Foundation of Korea (Grant No. 2010-0020413) and the Shanghai Pujiang Program Grant No. 21PJ1405200 . T.Z. is supported by the National Natural Science Foundation of China (NSFC) General Grant No. 12071139 and the Science and Technology Commission of Shanghai Municipality (Grant No. 18dz2271000).

\section{A Clifford type inequality}

In this section, we recall a Clifford type result in \cite{Yuan_Zhang} that will be used afterwards. All results in this section hold also in positive characteristics.

\subsection{$\varepsilon$ for divisors}
Let $V$ be a smooth variety of dimension $n > 0$ and let $L$ be a $\QQ$-divisor on $V$. For any big divisor $M$ on $V$ with $|M|$ base point free, take the smallest integer $\lambda_M > 0$ so that the divisor $\lambda_M M - L$ is pseudo-effective. When $n \ge 2$, we define
$$
\varepsilon(V, L, M):=(\lambda_M + 1)^{n-1} M^n.
$$
When $n=1$, we simply set 
$$
\varepsilon(V, L, M) = 1.
$$ 
For any $n > 0$, define
$$
\varepsilon(V, L) := \inf_M \varepsilon(V, L, M),
$$
where the infimum is taken over all divisors $M$ on $V$ chosen as above. 
In particular, when $n=1$, we have 
$$
\varepsilon(V, L) = 1.
$$

It is straightforward to check that 

\begin{prop} \label{prop: epsilon}
	The above $\varepsilon$ satisfies the following properties:
	\begin{itemize}
		\item [(1)] If $L' \ge L$, then $\varepsilon(V, L', M) \ge \varepsilon(V, L, M)$ for any $M$ chosen as above. In particular, $\varepsilon(V, L') \ge \varepsilon(V, L)$.
		\item [(2)] Let $\sigma: V' \to V$ be a birational morphism. Then $\varepsilon(V', \sigma^*L) \le \varepsilon(V, L)$.
	\end{itemize}
\end{prop}

\subsection{A Clifford type inequality}
The main result in this section is the following one, which will be used later in the proof of Theorem \ref{main1}.
\begin{theorem} \label{thm: relative noether} 
	Let $V$ be a smooth variety of dimension $n > 0$. Suppose that $L$ is a $\QQ$-divisor on $V$ such that $K_V-L$ is pseudo-effective. Then
	$$
	h^0(V, \rounddown{L}) \le \frac{1}{2n!} \vol (L) + n \varepsilon(V, L).
	$$
\end{theorem}

\begin{proof}
	By \cite[Theorem 1.2]{Yuan_Zhang} which was stated only for integral divisors, we have
	$$
	h^0(V, \rounddown{L}) \le \frac{1}{2n!} \vol (\rounddown{L}) + n \varepsilon(V, \rounddown{L}).
	$$
	Note that $\vol(\rounddown{L}) \le \vol(L)$ and by Proposition \ref{prop: epsilon}, $\varepsilon(V, \rounddown{L}) \le \varepsilon(V, L)$. Thus the result follows easily.
\end{proof}

\begin{remark}
	As is explained in \cite{Yuan_Zhang}, Theorem \ref{thm: relative noether} is a natural generalization of the classical Clifford inequality.
\end{remark}

\section{Sharper estimate under extra assumptions}

To prove Theorem \ref{main2}, we need some estimates on the dimension of  $H^0(V, L)$ similar to Theorem \ref{thm: relative noether} but stronger. All the sharper bounds in this section are inspired by the work of Barja, Pardini, and Stoppino in \cite{Barja_Pardini_Stoppino}, where they proved the so-called ``continuous" estimates. However, under our setting we need explicit results instead, and the method we are going to employ is based on \cite{Yuan_Zhang,Zhang_Severi,Zhang_Slope}.

\subsection{A filtration for nef divisors} 

Let $f: V \to B$ be a fibration from a smooth variety $V$ of dimension $n$ to a smooth curve $B$ with a general fiber $F$. Let $L$ be a nef divisor on $V$. We first recall the following theorem.

\begin{theorem}  \cite[Theorem 4.1]{Zhang_Slope} \label{thm: filtration}
	Let $f: V \to B$, $F$ and $L$ be as above. Then there is a birational morphism $\sigma: V_L \to V$ and a sequence of triples
	$$
	\{(L_i, Z_i, a_i) | i=0, 1, \cdots, N\}
	$$
	on $V_L$ with the following properties:
	\begin{itemize}
		\item $(L_0, Z_0, a_0)=(\sigma^*L, 0, \nti_{f_L}(L_0))$ where $f_L: V_L \stackrel{\sigma}{\to} V \stackrel{f}{\to} B$ is the induced fibration. 
		\item For any $i=0, \cdots, N-1$, there is a decomposition
		$$
		|L_i-a_iF_L|= |L_{i+1}| + Z_{i+1}
		$$
		such that $Z_{i+1} \ge 0$ is the fixed part of $|L_i-a_iF_L|$ and that the movable part $|L_{i+1}|$ of $|L_i-a_iF_L|$ is base point free. Here $F_L = \sigma^*F$ denotes a general fiber of $f_L$, and $a_{i} = \nti_{f_L}(L_{i})$.
		\item We have $h^0(V_L, L_N-a_NF_L)=0$.
	\end{itemize}
\end{theorem}

In the above theorem, for any $0 \le i \le N$, the number $\nti_{f_L}(L_i)$ is defined by
$$
\nti_{f_L}(L_i) := \min \{a \in \ZZ| L_i - aF_L \,\mbox{is not nef}\}.
$$
Thus via Theorem \ref{thm: filtration}, we obtain a filtration
$$
\sigma^* L = L_0 > L_1 > \cdots > L_N \ge 0
$$
of nef divisors on a birational model $V_L$ of $V$. For simplicity, we still denote by $F$ a general fiber of $f_L: V_L \to B$ in the rest of this section.

\begin{prop} \label{prop: numerical inequality} \cite[Proposition 2.2]{Yuan_Zhang}
	We have the following two inequalities:
	\begin{align*}
		h^0(V, L) & \le \sum_{i=0}^{N} a_i h^0(F, L_i|_F); \\
		L^n & \ge n \sum_{i=0}^{N} a_i (L_i|_F)^{n-1} - n(L_0|_F)^{n-1}.
	\end{align*}
\end{prop}

\begin{prop} \label{prop: sum of ai} \cite[Lemma 2.3]{Yuan_Zhang}
	We have
	$$
	L_0^n \ge \left(\sum_{i=0}^{N} a_i - 1 \right) (L_0|_F)^{n-1}.
	$$
\end{prop}

\subsection{Sharper bound involving the subcanonicity} \label{subcanonicity}
Let $V$ be a smooth variety of dimension $n > 0$ with the Kodaira dimension $\kappa(V) \ge 0$, and let $L$ be a $\QQ$-divisor on $V$. Let $M$ be a big divisor on $V$ such that $|M|$ is base point free. We recall that the numerical subcanonicity of $L$ with respect to $M$ is defined in \cite[Definition 5.1]{Barja_Pardini_Stoppino} as follows:
$$
r(L, M):= \frac{LM^{n-1}}{K_VM^{n-1}}.
$$
When $n=1$, set $r(L, M) = \frac{\deg L}{\deg K_V}$. When $K_VM^{n-1} = 0$, we have $\kappa(V) = 0$. In this case, we set $r(L, M) = + \infty$. Define a function $\delta$ as follows:
$$
\delta(x) = \left\{
    \begin{array} {ll}
        2, & x \le 1; \\
        \frac{2x}{2x-1}, & x > 1.
    \end{array}
    \right.
$$

\begin{theorem} \label{thm: subcanonicity}
	Let $L$ and $M$ be as above, and write $r = r(L, M)$. Then
	$$
	h^0(V, \rounddown{L}) \le \frac{1}{\delta(r)n!} \vol(L) + n \varepsilon(V, L, M)
	$$
\end{theorem}

\begin{proof}
	The proof is by induction, and we present it in several steps.
	
	Notice that the required inequality holds trivially if $h^0(V,\rounddown{L})=0$. We may make assumption $h^0(V, \rounddown{L})>0$ from now on.
	
	\textbf{Step 1}: Reduce to the case when $L$ is nef. 
	
	In fact, by replacing $V$ by an appropriate blowing up, we may assume that 
	$$
	L = L' + Z,
	$$
	where $L'$ is the movable part of $|\rounddown{L}|$ and $Z$ is its fixed part. It is clear that 
	$$
	r(L, M) \ge r(L', M), \quad \vol(L) \ge \vol(L'), \quad \varepsilon(V, L, M) \ge \varepsilon(V, L', M).
	$$ 
	Thus it suffices to prove Theorem \ref{thm: subcanonicity} for $L'$. 
	
	From now on, we assume that $L$ is a nef divisor.
	
	\textbf{Step 2}: The $n=1$ case.
	
	When $n=1$, Theorem \ref{thm: subcanonicity} is straightforward. If $h^1(V, L) \ne 0$, the classical Clifford inequality implies Theorem \ref{thm: subcanonicity}. Otherwise, by the Riemann-Roch theorem, 
	$$
	h^0(V, L) = \deg L - \frac{1}{2} \deg K_V = \left(1 - \frac{1}{2r}\right) \deg L.
	$$
	Thus the proof is completed.
	
	\textbf{Step 3}: The proof when $L^n > 0$.
	
	Now we assume that Theorem \ref{thm: subcanonicity} holds for dimension $k < n$. Choose a general pencil in $|M|$ and blow up the indeterminacies of this pencil, denoted by $\pi: V_0 \to V$. We get a fibration
	$$
	f: V_0 \to \PP^1
	$$
	such that the general fiber $F$ of $f$ is isomorphic to a general member of the chosen pencil. By the adjunction, $\kappa(F) \ge 0$. Write $M_0 = \pi^*M$ and $L_0 = \pi^*L$. It follows that
	$$
	r(L, M) = \frac{L_0M_0^{n-2}F}{(\pi^*K_V)M_0^{n-2}F} \ge r(L_0|_F, M_0|_F),
	$$
	where the last inequality follows from the adjunction. 
	
	Apply Theorem \ref{thm: filtration} to $f$ and $L_0$. Replacing $V_0$ by a further blowing up if necessary, we get triples
	$$
	(L_i, Z_i, a_i) \quad (i=0, \ldots, N)
	$$
	on $V_0$, and $L_i$ and $a_i$ satisfy the inequalities in Proposition \ref{prop: numerical inequality} and \ref{prop: sum of ai}. Note that by the definition of $r(L, M)$, we see that
	$$
	r(L_i|_F, M_0|_F) \le r(L_0|_F, M_0|_F) = r.
	$$
	By induction and using the fact that the function $\delta$ is non-increasing, we have
	$$
	h^0(F, L_i|_F)  \le \frac{1}{\delta(r)(n-1)!} (L_i|_F)^{n-1} + (n-1) \varepsilon(F, L_i|_F, M_0|_F).
	$$
	Combine this with Proposition \ref{prop: numerical inequality}. It follows that
	\begin{align*}
		 h^0(V_0, L_0) - \frac{1}{\delta(r)n!} L_0^n 
		\le  (n-1) \sum_{i=0}^{N} a_i\varepsilon(F, L_i|_F, M_0|_F) + \frac{1}{(n-1)!}(L_0|_F)^{n-1}.
	\end{align*}
	
	To estimate the right hand side of the above inequality, let $\lambda$ be the smallest integer such that $\lambda M-L$ is pseudo-effective. Note that $L^n > 0$.
	\begin{itemize}
		\item [(1)] It implies that $L^n \le \lambda L^{n-1}M = \lambda (L_0|_F)^{n-1}$. In particular, $(L_0|_F)^{n-1} > 0$. Thus by Proposition \ref{prop: sum of ai},
		$$
		\sum_{i=0}^{N} a_i \le \frac{L_0^n}{(L_0|_F)^n} + 1 \le \lambda + 1.
		$$
		
		\item [(2)] By Proposition \ref{prop: epsilon} (1),
		$$
		\varepsilon(F, L_i|_F, M_0|_F) \le \varepsilon(F, L_0|_F, M_0|_F).
		$$
		Moreover, since $\lambda M_0|_F - L_0|_F$ is also pseudo-effective, we have 
		$$
		\varepsilon(F, L_0|_F, M_0|_F) \le (\lambda+1)^{n-2}(M_0|_F)^{n-1} = (\lambda+1)^{n-2}M^n.
		$$
		
		\item [(3)] We have 
		$$
		(L_0|_F)^{n-1} = L^{n-1}M \le \lambda L^{n-2}M^2 \le \cdots \le \lambda^{n-1}M^n
		$$
	\end{itemize}
    Combining all above inequalities, it follows that 
    \begin{align*}
    	h^0(V_0, L_0) - \frac{1}{\delta(r)n!} L_0^n  & \le (n-1)(\lambda+1)^{n-1}M^n + \frac{1}{(n-1)!} \lambda^{n-1}M^n  \\
    	& \le n \varepsilon(V, L, M). 
    \end{align*}
    Thus the proof in this case is completed.
    
    \textbf{Step 4}. The proof when $L^n = 0$.
    
    In this case, the proof is easier. Since $L$ is not big, we know that 
    $$
    h^0(V, L-M) = 0.
    $$
    Take $W$ to be a general member in $|M|$, and we have 
    $$
    h^0(V, L) \le h^0(W, L|_W).
    $$
    Therefore, by induction, we deduce that
    $$
    h^0(V, L) \le \frac{1}{(n-1)!} (L|_W)^{n-1} + (n-1)\varepsilon(W, L|_W, M|_W).
    $$
    
    Let $\lambda$ be the smallest integer such that $\lambda M - L$ is pseudo-effective. Similar to Step 3, we have
    \begin{itemize}
    	\item [(1)] $(L|_W)^{n-1} = L^{n-1}M \le \lambda^{n-1}M^n$;
    	\item [(2)] $\varepsilon(W, L|_W, M|_W) \le (\lambda+1)^{n-2}M^n$.
    \end{itemize}
    
    Combining the above inequalities, it follows that
    $$
    h^0(V, L) \le \frac{1}{(n-1)!} \lambda^{n-1}M^n + (n-1) (\lambda+1)^{n-2}M^n \le n \varepsilon(V, L, M).
    $$
    Thus the whole proof is completed.
\end{proof}

\subsection{Sharper bound involving the mapping degree}

Let $V$ be a smooth variety of dimension $n \ge 2$, and let $L$ be a $\QQ$-divisor on $V$ such that $K_V-L$ is pseudo-effective. Instead of the subcanonicity, we suppose that
$$
a: V \to \Sigma
$$
is a generically finite morphism onto a (possibly singular) variety $\Sigma$. Let $H$ be a sufficiently ample divisor on $\Sigma$, and write $M = a^*H$. The assumption will be used till the end of this section.

\subsubsection{Preparation} \label{sharp bound preparation}

We first assume that $V$ is a surface and $|L|$ is base point free. Though this assumption looks simple, all results we need can be reduced to this setting.

\begin{lemma} \label{lem: L-M=0}
	If $h^0(V, L-M) = 0$, then 
	$$
	h^0(L) \le \frac{1}{2}LM + 1 \le \varepsilon(V, L, M).
	$$
\end{lemma}

\begin{proof}
	Choose a general curve $C \in a^*|H|$. By Bertini's theorem, we may assume that $C$ is smooth. The assumption $h^0(V, L-M) = 0$ just tells us that $h^0(V, L) \le h^0(C, L|_C)$. Thus the first inequality is just a combination of the Clifford inequality and the Riemann-Roch theorem again.
	
	The second inequality is directly from the definition of $\varepsilon$. Actually, let $\lambda$ be the smallest integer such that $\lambda M - L$ is pseudo-effective. Then
	$$
	\frac{1}{2}LM + 1 \le \frac{\lambda}{2} M^2 + 1 \le (\lambda + 1) M^2 = \varepsilon(V, L, M).
	$$
	The proof is completed.
\end{proof}

Now suppose that $h^0(V, L-M) > 0$. Let 
$$
\gamma : = \max\{i \in \ZZ| h^0(V, L-iM) > 0 \}.
$$
Obviously, $\gamma \ge 1$. 

\begin{lemma} \label{lem: L^2 h^0(L) h^0(2L)}
	If $h^0(V, L-M) > 0$, then
	$$
	L^2 \ge h^0(V, 2L) - h^0(V, L) - 1.
	$$
\end{lemma}

\begin{proof}
	Take a general member $D \in |L|$. By assumption, $D$ is big. Thus we may assume that $D$ is smooth and irreducible. Consider the following exact sequence
	$$
	0 \to H^0(V, L) \to H^0(V, 2L) \to H^0(D, 2L|_D).
	$$
	Since $K_V - L$ is pseudo-effective, we know that $\deg(2L|_D) \le \deg (K_V|_D + L|_D) = \deg K_D$, i.e., $K_D - 2L|_D$ is pseudo-effective. Apply the Clifford inequality (when $h^1(D, 2L|_D) > 0$) or the Riemann-Roch theorem (when $h^1(D, 2L|_D) = 0$) for $2L|_D$, and it follows that
	$$
	L^2 = \frac{1}{2} \deg(2L|_D) \ge h^0(D, 2L|_D) - 1 \ge h^0(V, 2L) - h^0(V, L) - 1.
	$$
	The proof is completed.
\end{proof}

Let $C \in a^*|H|$ be a general member, hence smooth. Consider the following two restriction maps
$$
\mathrm{res}_{1, i}: H^0(V, L-iM) \to H^0(C, L|_C - iM|_C)
$$
and
$$
\mathrm{res}_{2, j}: H^0(V, 2L-jM) \to H^0(C, 2L|_C - jM|_C).
$$
The kernels of the above two maps are just $H^0(V, L-(i+1)M)$ and $H^0(V, 2L-(j+1)M)$, respectively. 

Let $V_{1, i}$ (resp. $V_{2, j}$) denote the image of $H^0(V, L-iM)$ (resp.  $H^0(V, 2L-jM)$) under $\mathrm{res}_{1, i}$ (resp. $\mathrm{res}_{2, j}$). 
\begin{lemma} \label{lem: h^0(L) h^0(2L)}
	We have 
	\begin{align*}
		h^0(V, L) & = \sum_{i=0}^{\gamma-1} \dim V_{1, i} + h^0(V, L-\gamma M) \\
		h^0(V, 2L) & = \sum_{j=0}^{2\gamma-1} \dim V_{2, j} + h^0(V, 2L-2\gamma M) \ge 2\sum_{i=0}^{\gamma-1} \dim V_{2, 2i} - \dim V_{2, 0}.
	\end{align*}
\end{lemma} 

\begin{proof}
	The two equalities are obvious. The last inequality in the second formula holds simply because $h^0(V, 2L-2\gamma M) > 0$ and $\dim V_{2, 2i-1} \ge \dim V_{2, 2i}$ for any $1 \le i \le \gamma-1$.
\end{proof}

Let $|N_i|$ denote the movable part of $|L-iM|$. Note that the base locus of $|N_i|$ is either empty or of dimension zero. We deduce that $N_i$ is nef. Also, we have 
$$
\dim V_{1, i} = \dim |N_i||_C+1.
$$

\begin{lemma} \label{lem: V_2i V_1i}
	For $0 \le i \le \gamma$, we have
	$$
	\dim V_{2, 2i} \ge 2 \dim V_{1, i} - 1.
	$$
	If moreover, the linear system $|N_i||_C$ induces a birational map on $C$, then
	$$
	\dim V_{2, 2i} \ge 3 (\dim V_{1, i} - 1).
	$$
\end{lemma}

\begin{proof}
	This is just \cite[Lemma 5.3]{Barja_Pardini_Stoppino} for $k=2$.
\end{proof}

In the following, we will apply the above results to deduce more inequalities subject to the degree of the map $a$. The notation here will be frequently used in the sequel.

\subsubsection{$\deg a = 1$} We first consider the case when $a$ is birational.

\begin{theorem} \label{thm: deg 1}
	Suppose that $\deg a = 1$ and that $K_V - L$ is pseudo-effective. Then we have
	$$
	h^0(V, \rounddown{L}) \le \frac{2}{5 n!} \vol(L) + n \varepsilon(V, L, M).
	$$
\end{theorem}

Similar to the proof of Theorem \ref{thm: subcanonicity}, we may assume that $L$ is nef. Actually, we may even assume that $|L|$ is base point free. Moreover, we only need to prove Theorem \ref{thm: deg 1} when $n = 2$ (i.e., Lemma \ref{lem: deg 1 surface}), and the general result follows by an inductive argument almost identical to Step 3 and Step 4 in the proof of Theorem \ref{thm: subcanonicity}.

One little difference is that, instead of choosing a general pencil in $|M|$ as in Step 3 of the proof of Theorem \ref{thm: subcanonicity}, here we choose a general pencil in the sub linear system $a^*|H| \subseteq |M|$. Since $a^*|H|$ is also base point free, the smoothness of a general member in it is guaranteed by Bertini's theorem. This adjustment will be used till the end of this section. Note that the restriction of $a$ on a general member of $a^*|H|$ has degree one. This is the key point for us to use the induction.

With this adjustment and by Lemma \ref{lem: L-M=0}, we eventually reduce Theorem \ref{thm: deg 1} to the following lemma.

\begin{lemma} \label{lem: deg 1 surface}
	Theorem \ref{thm: deg 1} holds when $n=2$, $|L|$ is base point free and $h^0(V, L-M) > 0$.
\end{lemma}

\begin{proof}
	We claim that
	\begin{equation} \label{eq: deg 1 claim}
		h^0(V, 2L) - 6h^0(V, L) \ge -8LM - 7. 
	\end{equation}
	Suppose the claim holds. Together with Lemma \ref{lem: L^2 h^0(L) h^0(2L)}, we deduce that
	$$
	h^0(V, L) \le \frac{1}{5} L^2 + \frac{8}{5}(LM + 1),
	$$
	and the proof will be completed just noting that 
	$$
	 \frac{8}{5}(LM + 1) < 2\varepsilon(V, L, M)
	$$
	just as in the proof of Lemma \ref{lem: L-M=0}.
	
	To prove the claim, let $C$, $\gamma$, $V_{1, i}$, $V_{2, j}$ be the same as in \S \ref{sharp bound preparation}. For $0 \le i \le \gamma-1$, $|M|$ is a sub linear system of $|L-iM|$, which means that $|M||_C$ is a sub linear system of $V_{1, i}$. Note that $|M||_C$ induces a birational map from $C$. We deduce that the map induced by $V_{1, i}$ $(0 \le i \le \gamma-1)$ is birational. Thus it follows from Lemma \ref{lem: h^0(L) h^0(2L)} and the second inequality in Lemma \ref{lem: V_2i V_1i} that
	$$
	h^0(V, 2L) - 6h^0(V, L) \ge - 6 \left(\gamma + h^0(V, L-\gamma M)\right) - \dim V_{2,0}.
	$$
	Let us estimate the right hand side of the above inequality. 
	\begin{itemize}
		\item [(1)] Since $h^0(V, L-(\gamma+1)M) = 0$, by Lemma \ref{lem: L-M=0}, we have
		$$
		h^0(V, L-\gamma M) \le \frac{1}{2}(LM - \gamma M^2) + 1 \le LM - \gamma M^2 + 1 
		$$ 
		In particular, 
		$$
		h^0(V, L-\gamma M) + \gamma \le LM + 1
		$$
		\item [(2)] Note that $\dim V_{2, 0} \le h^0(C, 2L|_C)$. By the Clifford inequality and the Riemann-Roch theorem similar as before, we simply deduce that
		$$
		\dim V_{2, 0} \le h^0(C, 2L|_C) \le \deg(2L|_C) + 1 = 2LM + 1.
		$$
	\end{itemize}
    Combining the above two inequalities together, we prove the claim.
\end{proof}

\subsubsection{$a$ is not composed with an involution}
Second, we consider the case when $a$ is not composed with an involution. That is, there is no generically finite map $V \dashrightarrow V'$ of degree two through which $a$ factors birationally.

\begin{theorem} \label{thm: deg 3}
	Suppose that $a$ is not composed with an involution and that $K_V - L$ is pseudo-effective. Then we have
	$$
	h^0(V, \rounddown{L}) \le \frac{4}{9 n!} \vol(L) + n \varepsilon(V, L, M).
	$$
\end{theorem}

Similar as we did for Theorem \ref{thm: deg 1}, we may assume that $n=2$, $|L|$ is base point free, and $h^0(V, L-M) > 0$. For general $n$, we just use the induction. Note that by our assumption, the restriction of $a$ on a general member of $a^*|H|$ is not composed with an involution, either. See  \cite[Proposition 2.8]{Barja_Pardini_Stoppino} for example. This guarantees that the inductive argument also works in this situation. Therefore, Theorem \ref{thm: deg 3} boils down to the following lemma.

\begin{lemma} \label{lem: deg 3 surface}
	Theorem \ref{thm: deg 3} holds when $n=2$, $|L|$ is base point free, and $h^0(V, L-M) > 0$.
\end{lemma}

\begin{proof}
	We sketch the proof here since it is similar to that of Lemma \ref{lem: deg 1 surface}.
	
	Let $C$, $\gamma$, $V_{1, i}$, $V_{2, j}$, $N_i$ be identical to those in \S \ref{sharp bound preparation}. Let
	$$
	i_0 = \min \{0 \le i \le \gamma-1 | V_{1, i} \,\mbox{does not induce a birational map on $C$}\}.
	$$
	With this notation, using the same strategy as for proving \eqref{eq: deg 1 claim}, we deduce that
	$$
	h^0(V, 2L) - 6h^0(V, L) \ge - 8LM - 7 - 2\sum_{i=i_0}^{\gamma-1} \dim V_{1, i}.
	$$
	Comparing to the proof of \eqref{eq: deg 1 claim}, the only modification we make here is that, for $i \ge i_0$, we have to use the first inequality in Lemma \ref{lem: V_2i V_1i} to compare $\dim V_{2, 2i}$ with $\dim V_{1, i}$, which is the reason for having an extra term  $-2\sum_{i=i_0}^{\gamma-1} \dim V_{1, i}$ on the right hand side. 
	
	Combining this inequality with Lemma \ref{lem: L^2 h^0(L) h^0(2L)}, it follows that
	\begin{equation} \label{eq: deg 3 1}
		L^2 \ge 5h^0(V, L) - 2\sum_{i=i_0}^{\gamma-1} \dim V_{1, i} - 8 (LM +1).
	\end{equation}

	On the other hand, recall that for any $0 \le i \le \gamma-1$, $N_i$ is nef and 
	$$
	\dim V_{1, i} = \dim |N_i||_C+1.
	$$ 
	Note that in the current setting, $N_0 = L$ and $|N_{i+1}|$ is also the movable part of $|N_i-M|$.
	
	For any $i > 0$, we have
	\begin{equation} \label{eq: deg 3 V_{1, i} 1}
		N_{i-1}^2 - N_i^2 \ge (N_{i-1} + N_i) M \ge 2 N_iM \ge 4 \dim V_{1, i} - 4,
	\end{equation}
	where the last inequality follows from the fact that $K_C - (L|_C - iM|_C)$ is pseudo-effective.
	%Since $K_C - (L|_C - iM|_C)$ is pseudo-effective, for any $i > 0$, we have
	%\begin{align}
		%\dim V_{1, i} &\le \frac{1}{2}N_iM + 1\nonumber \\ 
		%& \le \frac{1}{4}N_iM +\frac{1}{4}N_{i-1}M+ 1\nonumber\\
		%&\le \frac{1}{4}N_i(N_{i-1}-N_i)+\frac{1}{4}N_{i-1}(N_{i-1}-N_i) %+1\nonumber\\
		%&\le \frac{1}{4}(N_{i-1}^2 - N_i^2) + 1. \label{eq: deg 3 V_{1, i} 1}
	%\end{align}
	When $i \ge i_0$, $V_{1, i}$ induces a map on $C$ of degree at least three. Otherwise, the map $\phi_{|L-iM|}$ induced by the linear system $|L-iM|$ would factor through a degree two map from $V$, and $a$ would factor through $\phi_{|L-iM|}$, which is a contradiction. Let 
	$$
	\phi_i: C \to C'_i
	$$
	be the morphism induced by the movable part of $V_{1, i}$. Then $\deg \phi_i \ge 3$. Since $\phi_i$ factor through the normalization of $C'_i$, we may assume that the curve $C'_i$ is normal, hence smooth. Then
	$$
	|N_i||_C = \phi_i^*|L'_i| + Z'_i,
	$$
	where $L'_i$ and $Z'_i$ are effective divisors on $C'$. Since 
	$$
	\dim V_{1, i} \, \le \, h^0(C'_i, L'_i) \le \deg L'_i + 1 \le \frac{1}{\deg \phi_i} N_iM + 1,
	$$ 
	similar to \eqref{eq: deg 3 V_{1, i} 1}, we deduce that for $i \ge \max\{1, i_0\}$,
	\begin{align}
		N_{i-1}^2 - N_i^2 \ge 2N_i M \ge 6 \dim V_{1, i} - 6. \label{eq: deg 3 V_{1, i} 2}
	\end{align}
	%\begin{align}
	    %\dim V_{1, i}  \le \frac{1}{3} N_iM + 1 \le \frac{1}{6} \left(N_{i-1}^2 - N_i^2\right) + 1. \label{eq: deg 3 V_{1, i} 2}
	%\end{align}
	Note that we also have
	$$
	\dim V_{1, 0} \le \left\{ 
		\begin{array} {ll}
		    \frac{1}{2}LM + 1, & i_0 > 0; \\
		    & \\
		    \frac{1}{3}LM + 1, & i_0 = 0.
		\end{array}
		\right.
	$$ 
	Together with \eqref{eq: deg 3 V_{1, i} 1} and \eqref{eq: deg 3 V_{1, i} 2} for all $i > 0$, we deduce that
	\begin{align}
		L^2 & = \sum_{i=1}^{\gamma-1} (N_{i-1}^2 - N_i^2) + N_{\gamma-1}^2 \nonumber \\ 
		& \ge 4 \sum_{i=0}^{i_0-1} \dim V_{1, i} + 6 \sum_{i=i_0}^{\gamma-1} \dim V_{1, i} - 2LM - 6\gamma + N_{\gamma-1}^2  \nonumber\\
		& \ge 4h^0(V, L) + 2 \sum_{i=i_0}^{\gamma-1} \dim V_{1, i} - 4h^0(V, L-\gamma M) - 2LM - 6\gamma \nonumber \\
		& \ge 4h^0(V, L) + 2 \sum_{i=i_0}^{\gamma-1} \dim V_{1, i} - 4LM - 6 \gamma - 4. \label{eq: deg 3 2}
	\end{align}
	The third inequality here is due to Lemma \ref{lem: h^0(L) h^0(2L)}. For the last inequality, by Lemma \ref{lem: L-M=0} and the definition of $\gamma$, we have
	$$
	h^0(V, L-\gamma M) \le \frac{1}{2} (LM - \gamma M^2) + 1.
	$$
	Then it is easy to deduce that
	$$
	4h^0(V, L-\gamma M) \le 2LM - 2\gamma M^2 + 4.
	$$
	Thus \eqref{eq: deg 3 2} is verified.
	
	Now adding \eqref{eq: deg 3 1} and \eqref{eq: deg 3 2} together, it follows that
	$$
	2L^2 \ge 9h^0(V, L) - 12 LM - 6\gamma - 12,
	$$
	i.e.,
	$$
	h^0(V, L) \le \frac{2}{9} L^2 + \frac{4}{3} LM + \frac{2}{3} \gamma + \frac{4}{3}.
	$$
	
	Finally, let $\lambda$ be the smallest integer such that $\lambda M-L$ is pseudo-effective. Noting that $\gamma \le \lambda$, we deduce that
	$$
	\frac{4}{3} LM + \frac{2}{3} \gamma + \frac{4}{3} \le \frac{4}{3} \lambda M^2 + \frac{2}{3} \lambda + \frac{4}{3} \le 2(\lambda+1)M^2 = 2 \varepsilon(V, L, M).
	$$
	Thus the whole proof of this lemma is completed.
\end{proof}

\subsubsection{$a$ is composed with an involution and $\kappa(\Sigma) > 0$}
Finally, we consider the case when $a$ is composed with an involution and $\Sigma$ is birational to a smooth projective variety of positive Kodaira dimension. Let $\pi: \Sigma' \to \Sigma$ be a resolution of singularities of $\Sigma$. Then $\kappa(\Sigma') > 0$. Set
$$
r'(L, M, \Sigma') := \frac{LM^{n-1}}{2K_{\Sigma'}(\pi^*H)^{n-1}}
$$
By the assumption, $K_{\Sigma'}(\pi^*H)^{n-1} > 0$. Thus $r'(L, M, \Sigma') < \infty$.

\begin{theorem} \label{thm: deg 2}
	Let the notation be as above. Write $r' =  r'(L, M, \Sigma')$. Suppose that $K_V-L$ is pseudo-effective. Then we have 
    $$
    h^0(V, \rounddown{L}) \le \frac{2 \delta(r') - 1}{(5 \delta(r') - 3) n!} \vol(L) + n \varepsilon(V, L, M).
    $$
    Moreover, for any $\QQ$-divisor $L_1 \le L$, we have 
    $$
    h^0(V, \rounddown{L_1}) \le \frac{2 \delta(r') - 1}{(5 \delta(r') - 3) n!} \vol(L_1) + n \varepsilon(V, L, M).
    $$
\end{theorem}

Here the function $\delta(x)$ is the same as that in Theorem \ref{thm: subcanonicity}. Note that under this setting, $\delta(r') > 1$. Moreover, since $r'_1:=r'(L_1, M, \Sigma') \le r'$, we have $\delta(r'_1) \ge \delta(r')$ and $\frac{2 \delta(r'_1) - 1}{5 \delta(r'_1) - 3} \le \frac{2 \delta(r') - 1}{5 \delta(r') - 3}$. Therefore, the second inequality in Theorem \ref{thm: deg 2} can be deduced from the first one for $L_1$.

Note that the restriction of $a$ on a general member of $a^*|H|$ is composed with an involution. Furthermore, by the adjunction, a smooth model of a general member of $|H|$ has positive Kodaira dimension. Thus the induction method works here, and Theorem \ref{thm: deg 2} is finally reduced to the following result.

\begin{lemma} \label{lem: deg 2 surface}
	Theorem \ref{thm: deg 2} holds when $n=2$, $|L|$ is base point free, and $h^0(V, L-M) > 0$.
\end{lemma}

\begin{proof}
	The proof is just a modification of the proof of Lemma \ref{lem: deg 3 surface}. We sketch it and leave the details to the interested reader.
	
	Let $C$, $\gamma$, $V_{1, i}$, $V_{2, j}$, $N_i$ and $i_0$ be identical to
	those in the proof of Lemma \ref{lem: deg 3 surface}. Then it is easy to see that \eqref{eq: deg 3 1} still holds here, i.e., 
	\begin{equation} \label{eq: deg 2 1}
		L^2 \ge 5h^0(V, L) - 2\sum_{i=i_0}^{\gamma-1} \dim V_{1, i} - 8 (LM +1).
	\end{equation}
	For any $1 \le i \le \gamma-1$, \eqref{eq: deg 3 V_{1, i} 1} holds also here, i.e.,
	\begin{equation} \label{eq: deg 2 V_{1, i} 1}
		\dim V_{1, i} \le \frac{1}{4} (N_{i-1}^2 - N_i^2) + 1.
	\end{equation}
	
	The major modification is a replacement of \eqref{eq: deg 3 V_{1, i} 2}. For $i_0 \le i \le \gamma - 1$, $V_{1, i}$ induces a map on $C$ of degree at least two. Let $\phi_i: C \to C'_i$, $L'_i$ and $Z'_i$ be as in the proof of Lemma \ref{lem: deg 3 surface}. We may further assume that the curve $C'_i$ is normal. 
	By Theorem \ref{thm: subcanonicity} and the fact that $\deg \phi_i \ge 2$, we deduce that
	$$
	\dim V_{1, i} \, \le \, h^0(C'_i, L'_i) \le \frac{1}{\delta(r'_i)} \deg L'_i + 1 \le \frac{1}{2\delta(r'_i)} N_iM + 1,
	$$
	where $r'_i = \frac{\deg L'_i}{\deg K_{C'_i}}$. Now we claim that
	$$
	\delta(r'_i) \ge \delta(r')
	$$
	for any $i \ge i_0$ as above. With this claim, we deduce that for $i \ge \max \{1, i_0\}$,
	\begin{equation} \label{eq: deg 2 V_{1, i} 2}
		\dim V_{1, i} \le \frac{1}{4\delta(r')} (N_{i-1}^2 - N_i^2) + 1.
	\end{equation}
	
	To prove the claim, we only need to prove that $r'_i \le r'$. Since we already have $\deg L'_i \le \frac{1}{2} LM$ as above, it suffices to prove that $\deg K_{C'_i} \ge K_{\Sigma'}(\pi^*H)$. This is rather obvious. The key is to note that $a|_C$ factors through $\phi_i$. Via this factorization, $C'_i$ maps to a general curve in $|H|$ on $\Sigma$. Since $\pi^*|H|$ is base point free, by Bertini's theorem, a general member of $\pi^*|H|$ is smooth. Moreover, the aforementioned map on $C'_i$ lifts to a map from $C'_i$ to a general member $C''\in \pi^*|H|$. Therefore, by the Hurwitz formula and the adjunction formula,
	$$
	\deg K_{C'_i} \ge \deg K_{C''} = K_{\Sigma'}(\pi^*H) + (\pi^*H)^2 > K_{\Sigma'}(\pi^*H).
	$$
	Thus the claim is verified, and \eqref{eq: deg 2 V_{1, i} 2} is established.
	
	Having the above modification, we can proceed the proof as before. Sum up \eqref{eq: deg 2 V_{1, i} 1} and \eqref{eq: deg 2 V_{1, i} 2} over all the above $i > 0$. Note that 
	$$
	\dim V_{1, 0} \le \left\{ 
	    \begin{array} {ll}
	        \frac{1}{2}LM + 1, & i_0 > 0; \\
	        & \\
	        \frac{1}{2 \delta(r')}LM + 1, & i_0 = 0.
	    \end{array}
	    \right.
	$$ 
	It follows that
	\begin{align}
		L^2 & \ge 4 \sum_{i=0}^{i_0-1} \dim V_{1, i} + 4\delta(r')  \sum_{i=i_0}^{\gamma-1} \dim V_{1, i} - 2LM - 4 \delta(r')\gamma + N_{\gamma-1}^2 \nonumber \\
		 & \ge 4h^0(V, L) + 4(\delta(r') - 1) \sum_{i=i_0}^{\gamma-1} \dim V_{1, i} - 4h^0(V, L-\gamma M) \nonumber \\ & \quad - 2LM - 4 \delta(r') \gamma. \nonumber
	\end{align}
	Using the argument for proving \eqref{eq: deg 3 2}, we can similarly deduce that 
	$$
	4h^0(V, L-\gamma M) + 2LM + 4 \delta(r') \gamma  \le 4LM + 4 \delta(r')\gamma + 4.
	$$
	The above two inequalities imply that
	\begin{equation} \label{eq: deg 2 2}
		L^2 \ge 4h^0(V, L) + 4(\delta(r') - 1) \sum_{i=i_0}^{\gamma-1} \dim V_{1, i} - 4LM - 4 \delta(r') \gamma - 4.
	\end{equation}
	
	For simplicity, we just write $\delta = \delta(r')$. As before, we use \eqref{eq: deg 2 1} and \eqref{eq: deg 2 2} together to eliminate $\sum_{i=i_0}^{\gamma-1} \dim V_{1, i}$. It follows that
	$$
	(2 \delta- 1) L^2 \ge (10 \delta- 6) h^0(V, L) - (16 \delta - 12)LM - 4 \delta \gamma - (16 \delta - 12),
	$$
	i.e.,
	$$
	h^0(V, L) \le \frac{2 \delta - 1}{10 \delta- 6} L^2 + \frac{8 \delta - 6}{5 \delta - 3} LM + \frac{2 \delta}{5 \delta - 3} \gamma + \frac{8 \delta - 6}{5 \delta - 3}.
	$$
	Since $1 < \delta \le 2$, it is straightforward to check that the above inequality implies that
	\begin{equation} \label{eq: deg 2 case 2}
		h^0(V, L) \le \frac{2 \delta - 1}{10 \delta - 6} L^2 + \frac{10}{7}LM + \gamma +  \frac{10}{7},
	\end{equation}
	Once again, let $\lambda$ be the smallest integer such that $\lambda M - L$ is pseudo-effective. Since $M^2 = (\deg a)H^2 \ge 2$ and $\gamma \le \lambda$, we deduce that
	$$
	\frac{10}{7} LM + \gamma +  \frac{10}{7} \le \frac{10}{7} \lambda M^2 + \frac{1}{2} \lambda M^2 + \frac{10}{7} < 2(\lambda+1)M^2 = 2 \varepsilon(V, L, M).
	$$
	Thus the whole proof is completed.
\end{proof}

\section{Some results about $\chi_f$} \label{section: chi_f}

Let $f: X \to B$ be a fibration from a smooth variety $X$ to a smooth curve $B$ of genus $g$, with a general fiber $F$. Recall that
$$
\chi_f:= \chi(X, \omega_X) - \chi(B, \omega_B) \chi(F, \omega_F).
$$
The goal of this section is to list some results about this relative invariant. We always assume that $f$ is of maximal Albanese dimension. Denote by
$$
a: X \to A
$$
the Albanese map of $X$. Let $q = \dim A = h^1(X, \CO_X)$. The above notation will be used throughout this section.

\subsection{$\chi_f$ equals the degree of a twisted Hodge bundle}

The following result relates $\chi_f$ to the degree of a twisted Hodge bundle.
\begin{prop} \label{prop: chi_f degree}
	With the above notation, we have
	$$
	\chi_f = \deg f_*(\omega_{X/B} \otimes \CP),
	$$
	where $\CP$ is a general torsion element in $\Pic^0(X)$.\footnote{Here being general means that $\CP$ is not contained in a certain proper subvariety (usually called the cohomological jumping loci) of $\Pic^0(X)$.}
\end{prop}

\begin{proof}
	This result has been proved by Hacon and Pardini \cite[Theorem 2.4]{Hacon_Pardini} assuming $g(B) \ge 2$. In fact, this assumption can be removed. Here we give a slightly different proof which works for any curve $B$.
	
	By the assumption, $a|_F: F \to A$ is generically finite onto its image. Let $\CP \in \Pic^0(X)$ be a general torsion element. Applying exactly the proof of \cite[Corollary 2.3]{Hacon_Pardini}, we conclude that $f_* (\omega_{X/B} \otimes \CP)$ is a torsion free, hence a locally free sheaf on $B$ of rank $r = \chi(F, \omega_F)$. Still by \cite[Corollary 2.3]{Hacon_Pardini}, for any $i > 0$, 
	$$
	R^i f_* (\omega_{X/B} \otimes \CP) = 0.
	$$
	Together with the Leray spectral sequence, we know that for any $i \ge 0$,
	$$
	h^i(X, \omega_X \otimes \CP) = h^i(B, f_* (\omega_X \otimes \CP) ).
	$$
	In particular, 
	$$
	\chi(X, \omega_X) = \chi(X, \omega_X \otimes \CP) = \chi(B, f_*(\omega_X \otimes \CP)).
	$$
	Combine all above together and apply the Riemann-Roch theorem for $f_* (\omega_X \otimes \CP)$. It follows that
	\begin{align*}
		\deg f_*(\omega_{X/B} \otimes \CP) & =  \deg f_*(\omega_X \otimes \CP) - 2 \chi(F, \omega_F) \chi(B, \omega_B) \\
		& =  \chi (B, f_* (\omega_X \otimes \CP) ) - \chi(F, \omega_F) \chi(B, \omega_B) \\
		& =  \chi(X, \omega_X) - \chi(F, \omega_F) \chi(B, \omega_B) \\
		& =  \chi_f.
	\end{align*}
	Thus the proof is completed.
\end{proof}

\subsection{The degree of the Hodge bundle under \'etale covers} \label{subsection: etale cover}

In this subsection, we assume that $g > 0$. Thus $X$ itself is of maximal Albanese dimension.

Let $\mu_m: A \to A$ be the multiplication-by-$m$ map of $A$. Let $X_m = X \times_{\mu_m} A$. Since $a$ is the Albanese map, $X_m$ is irreducible. Let $J(B)$ be the Jacobian variety of $B$. By the abuse of notation, let $\mu_m: J(B) \to J(B)$ also denote the multiplication-by-$m$ map of $J(B)$, and let $B_m=B \times_{\mu_m} J(B)$. Thus we have the following commutative diagram:
$$
\xymatrix{
	A \ar@/^1.5pc/[rrr]_{\mu_m} \ar[d] & X_m \ar[l]^{a_m} \ar[d]_{f_m} \ar[r]_{\nu_m} & X \ar[r]_{a} \ar[d]_f & A \ar[d]_{h} \\
	J(B) \ar@/_1.5pc/[rrr]_{\mu_m} & B_m \ar[l] \ar[r]_{\sigma_m} & B \ar[r] & J(B)
}
$$
Now we claim that if $m$ is a sufficiently large prime number, the morphism
$$
f_m: X_m \to B_m
$$ 
is always a fibration, i.e., it has connected fibers. To see this, let $A_F = \ker h$, which is also an abelian variety. We may assume that up to a translation by a point in $J(B)$, $a(F)$ generates $A_F$. Thus the kernel of the map $(a|_F)^*: \Pic^0(A_F) \to \Pic^0(F)$ is finite. Thus for any integer $m$ coprime to the cardinality of this kernel, the general fiber of $f_m$ is irreducible.

\begin{prop} \label{prop: degree limit}
	With the above notation, we have
	$$
	\lim\limits_{m \, \mathrm{prime}, m \to \infty} \frac{\deg {f_m}_* \omega_{X_m/B_m}}{m^{2q}} = \chi_f.
	$$
\end{prop}

\begin{proof}
	From the above construction, we know that for any $m > 0$, the morphism $\sigma_m: B_m \to B$ is \'etale. By the projection formula, 
	$$
	{\sigma_m}_* \CO_{B_m} = \bigoplus_{\mathcal{P} \in T_m(B)} \CP,
	$$ 
	where $T_m(B) \subset \Pic^0(B)$ is the subgroup of all $m$-torsion line bundles on $B$. There is a natural injective group homomorphism
	$$
	f^*: T_m(B) \to T_m(X)
	$$
	given by the pull-back of $f$, where $T_m(X) \subset \Pic^0(X)$ is the subgroup of all $m$-torsion line bundles on $X$. Let $m$ be a sufficiently large prime number, and let $X'_m = X \times_B B_m$. Then we have the following commutative diagram: 
	$$
	\xymatrix{
		X_m \ar@/_1pc/[rdd]_{f_m} \ar@/^1pc/[rrd]^{\nu_m} \ar[rd]^{\nu'_m} & & \\
		& X'_m \ar[r]^{\sigma'_m} \ar[d]_{f'_m}  & X \ar[d]^f \\
		& B_m \ar[r]^{\sigma_m} & B
	}
    $$
    It is clear that $\nu'_m: X_m \to X'_m$ is a Galois cover with $\Gal(\nu'_m) \simeq \frac{T_m(X)}{f^*T_m(B)}$. Thus by the projection formula,
    $$
    {\nu'_m}_* \omega_{X_m/B_m} = \bigoplus_{\CQ + f^*T_m(B)} \omega_{X'_m/B_m} \otimes (\sigma'_m)^*\CQ.
    $$
    Here the summation runs over all cosets of $f^*T_m(B)$ in $T_m(X)$ (whose cardinality equals $m^{2q-2g}$), and $\CQ$ is any representative in each corresponding coset. Thus we have the following splitting:
	$$
	{f_m}_* \omega_{X_m/B_m} = {f'_m}_* \left( {\nu'_m}_* \omega_{X_m/B_m} \right) = \bigoplus_{\CQ + f^*T_m(B)} \sigma_m^*(f_* (\omega_{X/B} \otimes \CQ)).
	$$ 
    All the above imply particularly that
	\begin{align*}
	\deg {f_m}_* \omega_{X_m/B_m} & = \deg {\sigma_m}_* \left({f_m}_* \omega_{X_m/B_m} \right) \\ 
	& = \deg \left({\sigma_m}_* \left( {f_m}_* \omega_{X_m} \right) \otimes \omega_B^{-1} \right) \\
	& =  \deg \left( f_* \left({\nu_m}_* \omega_{X_m} \right) \otimes \omega_B^{-1} \right).
	\end{align*}
	On the other hand, by the projection formula, 
	$$
	{\nu_m}_* \omega_{X_m} = \bigoplus_{\CP \in T_m(X)} \omega_X \otimes \CP,
	$$
	Thus it follows that 
	$$
	\deg {f_m}_* \omega_{X_m/B_m} = \sum_{\CP \in T_m(X)} \deg f_* ( \omega_{X/B} \otimes \CP).
	$$
	
	Let $S_m(X) = \{\CP \in T_m(X) | \deg f_*(\omega_{X/B} \otimes \CP) = \chi_f \}$ be the subset of $T_m(X)$. By Proposition \ref{prop: chi_f degree}, we know that the set 
	$$
	\bigcup_{m \in \ZZ} \left(T_m(X) \backslash S_m(X)\right)
	$$
	is contained in a proper subvariety of $\Pic^0(X)$. In particular,
	$$
	\lim\limits_{m \to \infty} \frac{\# S_m(X)}{\# T_m(X)} = \lim\limits_{m \to \infty} \frac{\# S_m(X)}{m^{2q}} = 1.
	$$
	Note that $\deg f_*(\omega_{X/B} \otimes \CP)$ is always non-negative (e.g., see \cite{Hacon_Pardini}) and bounded from above independent of $m$.
	We deduce that
	$$
	\lim\limits_{m \to \infty} \frac{\deg {f_m}_* \omega_{X_m/B_m}}{m^{2q}} = \chi_f.
	$$
	Thus the proof is completed.
\end{proof}

\section{Slope inequalities for fibrations over curves}

In this section, we prove a slope inequality for fibrations over curves whose general fiber is a smooth variety of general type. Throughout this section, we always assume that 
$$
f: X \to B
$$ 
is a fibration from a smooth variety $X$ of dimension $n \ge 2$ to a smooth curve $B$. Denote by $F$ a general fiber of $f$. 

\subsection{Xiao's method}  \label{Xiao's method}
Here we review Xiao's method and list some inequalities deduced from it. Most of the following facts can be found in \cite{Xiao} when $n=2$ and in \cite{Ohno, Konno, Barja_Stoppino} for general $n \ge 2$.

Let $L$ be a nef $\QQ$-divisor on $X$. Let 
$$
0= \CE_0 \subsetneq \CE_1 \subsetneq \cdots \subsetneq \CE_m = f_* \CO_X(\rounddown{L})
$$
be the Harder-Narasimhan filtration of $f_* \CO_X(\rounddown{L})$. For any $0 \le i \le m$, set 
$$
r_i= \rank \CE_i, \quad  \mu_i
=\frac{\deg(\CE_i/\CE_{i-1})}{\rank (\CE_i / \CE_{i-1})}.
$$
Then we have
$$
\mu_1> \mu_2 > \cdots > \mu_m
$$
as well as 
\begin{align} \label{eq: deg}
\deg \CE_k = \sum_{i=1}^{k-1} r_i (\mu_i - \mu_{i+1}) +r_k \mu_{k}
\end{align}
for each $1\le k\le m$.

For each $1 \le i \le m$, consider the rational map $\phi_i: X \dashrightarrow \mathbb{P}_B(\mathcal{E}_i)$ associated to the evaluation morphism $f^*\CE_i \to \CO_X(\rounddown{L})$. We may choose a common blowing up $\sigma: Y \to X$ which resolves all indeterminacies of $\phi_i$. Denote by $F_1$ a general fiber of $f \circ \sigma: Y \to B$. Applying Xiao's method, we obtain a sequence of nef $\mathbb{Q}$-Cartier divisors  
$$
N_1\le N_2 \le \cdots \le N_m \le N_{m+1}:= \sigma^*L
$$
on $Y$. Here $N_i = (\phi_i \circ \sigma)^*H_{\CE_i} - \mu_iF_1$, where $H_{\CE_i}$ is a hyperplane section of $\PP_{B}(\CE_i)$. For each $1 \le i \le m$, $N_i|_{F_1}$ is Cartier, $h^0(F_1, N_i|_{F_1})=r_i$,  
$$
N_{i+1}\ge N_i+(\mu_i-\mu_{i+1})F_1,
$$
and 
$$
\sigma^*L\ge N_i+\mu_i F_1.
$$ 
In particular,  $\sigma^*L - \mu_1 F_1$ is pseudo-effective, and for $1\le i\le m-1$, we have
\begin{align*}
N_{i+1}^n\ge (N_i+(\mu_i-\mu_{i+1})F_1)^n\ge N_i^n+n(\mu_i-\mu_{i+1})(N_i|_{F_1})^{n-1}. 
\end{align*} 
Thus the following lemma follows easily by induction.

\begin{lemma} \label{lem: Xiao}
	Keep the same notation as above. Suppose that for some $1\le i\le m$, we have $\mu_i\ge 0$. Let $k:=\max\{i|\ 1\le i\le m \ \text{and}\ \mu_i \ge 0 \}$. Then we have
	$$
	L^n\ge n\sum_{i=1}^{k-1}(\mu_i-\mu_{i+1})(N_i|_{F_1})^{n-1}+n\mu_k(N_k|_{F_1})^{n-1}.
	$$
\end{lemma}

\begin{proof}
	Inductively using the above estimate, we have
	$$
	N_k^n \ge N_1^n + n\sum_{i=1}^{k-1}(\mu_i-\mu_{i+1})(N_i|_{F_1})^{n-1}\ge n\sum_{i=1}^{k-1}(\mu_i-\mu_{i+1})(N_i|_{F_1})^{n-1}.
	$$
	The last inequality holds since $N_1$ is nef. Notice that $\sigma^*L\ge N_k+\mu_k F_1$ and $\mu_k\ge 0$, we have 
	$$
	L^n \ge (N_k+\mu_k F_1)^n = N_k^n+n\mu_k (N_k|_{F_1})^{n-1}
	$$
	Thus the proof is completed by combining the above estimates together.
\end{proof}

\subsection{A basic slope inequality} We have the following result.

\begin{prop}\label{prop: slope}
	Let $f: X \to B$ and $F$ be as before. Suppose that $L$ is a nef $\QQ$-divisor on $X$ such that $L|_F$ is big and that $K_F - L|_F$ is pseudo-effective. Then we have
	$$
	\left(1 + \frac{2n!(n-1) \varepsilon(F, L|_F)}{(L|_F)^{n-1}}\right) L^n \ge 2n! \deg f_* \CO_X(\rounddown{L}).
	$$
\end{prop}

\begin{proof}
	The inequality holds trivially when $\deg f_*\CO_X(\rounddown{L})\le 0$. Thus we may assume that $\deg f_*\CO_X(\rounddown{L})>0$.
	
	Let
	$$
	0= \CE_0 \subsetneq \CE_1 \subsetneq \cdots \subsetneq \CE_m = f_* \CO_X(\rounddown{L})
	$$
	be the Harder-Narasimhan filtration of $f_* \CO_X(\rounddown{L})$. Keep the same notation as in \S \ref{Xiao's method}. Since $\deg f_*\CO_X(\rounddown{L})>0$, we have $\mu_i>0$ for some $1\le i\le m$. Let $k:=\max\{i|\ 1 \le i \le m \ \text{and}\ \mu_i\ge 0 \}$. We have
	\begin{align}\label{eq: degree of Ek}
	    \deg \CE_k\ge \deg f_*\CO_X(\rounddown{L}).
	\end{align}
	
	By \eqref{eq: deg} and Lemma \ref{lem: Xiao}, we have the following two inequalities:
	\begin{align*}
	    L^n & \ge n \sum_{i=1}^{k-1} (\mu_i-\mu_{i+1})(N_i|_{F_1})^{n-1}+n\mu_k(N_k|_{F_1})^{n-1}, \\
		\deg \CE_k & = \sum_{i=1}^{k-1} r_i (\mu_i - \mu_{i+1})+r_k\mu_k.
	\end{align*}
	On the other hand, note that $N_i|_{F_1} \le \sigma^*L|_{F_1}$ for any $1 \le i \le m$ and that $K_{F_1} -\sigma^*L|_{F_1} \ge \sigma^*(K_F - L|_F)$ is pseudo-effective. By Theorem \ref{thm: relative noether} and Proposition \ref{prop: epsilon}, we have
	\begin{equation} \label{eq: key equation}
		r_i \le \frac{1}{2(n-1)!} (N_i|_{F_1})^{n-1} + (n-1) \varepsilon(F, L|_F).
	\end{equation}
	Combine the above three (in)equalities. We deduce that
	\begin{align*}
		L^n & \ge 2n! \deg \CE_k - 2n!(n-1) \varepsilon(F, L|_F) (\sum_{i=1}^{k-1} (\mu_i - \mu_{i+1})+\mu_{k}) \\
		& = 2n! \deg \CE_k - 2n!(n-1) \varepsilon(F, L|_F) \mu_1\\
		&\ge 2n!\deg f_*\CO_X(\rounddown{L})-2n!(n-1) \varepsilon(F, L|_F) \mu_1,
	\end{align*}
	where the last inequality follows by \eqref{eq: degree of Ek}.
	
	What is left to us is to estimate $\mu_1$. Note that $\sigma^*L - \mu_1 F_1$ is pseudo-effective. Thus
	$$
	L^n = (\sigma^*L)^n \ge \mu_1 (\sigma^*L|_{F_1})^{n-1} = \mu_1 (L|_F)^{n-1}.
	$$
	As a result, we deduce that
	$$
	\left(1 + \frac{2n!(n-1) \varepsilon(F, L|_F)}{(L|_F)^{n-1}}\right) L^n \ge 2n! \deg f_* \CO_X(\rounddown{L}).
	$$
	Thus the proof is completed.
\end{proof}

Before going further, we would like to remark that the inequality in Proposition \ref{prop: slope} is \emph{by no means} sharp. For example, when $n=2$, $f$ is a relatively minimal fibration by curves of genus $g \ge 2$, and $L = K_{X/B}$ (in this case $\varepsilon(F, L|_F) = 1$), Proposition \ref{prop: slope} yields
$$
K_{X/B}^2 \ge \left(\frac{4g-4}{g+1}\right) \deg f_* \omega_{X/B},
$$
which is weaker than the optimal slope inequality with the slope $\frac{4g-4}{g}$. This is because our estimate is not so delicate as Xiao's original version in \cite{Xiao} which also considers the intersection number contributed by the horizontal part $N_i|_{F_1}-N_{i+1}|_{F_1}$. See the proof of \cite[Lemma 2]{Xiao} for details. In other words, we have not employed Xiao's method in its full strength. However, Proposition \ref{prop: slope} is already enough to deduce Theorem \ref{main1}. Moreover, instead of using Theorem \ref{main1}, Proposition \ref{prop: slope} is sufficient for us to run the argument as in \cite[Proposition 4.4]{Barja_Stoppino} to deduce the absolute Severi inequality.

\subsection{Sharper slope inequalities}

In the following, we assume that
$$
a: F \to \Sigma
$$
is a generically finite map onto a projective variety $\Sigma$. Let $H$ be a sufficiently ample divisor on $\Sigma$. Let $M = a^*H$.

\begin{prop}\label{prop: sharper slope}
	Let $f: X \to B$ and $F$ be as before. Suppose that $L$ is a nef $\QQ$-divisor on $X$ such that $L|_F$ is big and that $K_F - L|_F$ is pseudo-effective.
	
	\begin{itemize}
		\item [(1)] If $a$ is birational, then
		$$
		\left(1 + \frac{5n!(n-1) \varepsilon(F, L|_F, M)}{2(L|_F)^{n-1}}\right) L^n \ge \frac{5n!}{2} \deg f_* \CO_X(\rounddown{L}).
		$$
		\item [(2)] If $a$ is not composed with an involution, then
		$$
		\left(1 + \frac{9n!(n-1) \varepsilon(F, L|_F, M)}{4(L|_F)^{n-1}}\right) L^n \ge \frac{9n!}{4} \deg f_* \CO_X(\rounddown{L}).
		$$
		\item [(3)] If $a$ is composed with an involution and $\Sigma$ has a smooth model of positive Kodaira dimension, then
		$$
		\left(1 + \frac{(5\delta(r') - 3) n!(n-1) \varepsilon(F, L|_F, M)}{(2 \delta(r') - 1)(L|_F)^{n-1}}\right) L^n \ge \frac{(5\delta(r') - 3) n!}{2 \delta(r') - 1} \deg f_* \CO_X(\rounddown{L}).
		$$
		Here $r'$ and $\delta$ are the same as in Theorem \ref{thm: deg 2}.
	\end{itemize}
\end{prop}

\begin{proof}
	The proof is almost identical to Proposition \ref{prop: slope}. We only need to replace \eqref{eq: key equation} by the inequalities in Theorem \ref{thm: deg 1}, \ref{thm: deg 3}, and the second inequality in Theorem \ref{thm: deg 2}, respectively. Then the results will follow. We leave the details to the interested reader.
\end{proof}

\section{Proof of the main theorems}

In the final section, we prove the main theorems of this paper. We always assume that $f: X \to B$ is a relatively minimal fibration from a variety $X$ of dimension $n \ge 2$ to a smooth curve $B$ with a general fiber $F$ and that $f$ is of maximal Albanese dimension. Let
$$
a: X \to A
$$
be the Albanese map of $X$. Write $q = h^1(X, \CO_X) = \dim A$.

\subsection{Preparation when $g(B)>0$} \label{subsection: notation}
Before proving the results, we list some notation that will be used throughout the section. We first assume that $g(B)>0$. Note that in this case, $X$ itself is of maximal Albanese dimension. 

Let $\pi: Y \to X$ be a resolution of singularities of $X$. Thus $Y$ is also of maximal Albanese dimension. Let 
$$
f': = f \circ \pi: Y \to B
$$ 
be the induced fibration with a general fiber $F'$, and let 
$$
b: Y \to A
$$ 
be the Albanese map of $Y$. 

Let $m$ be a sufficiently large prime number. Similar to \S \ref{subsection: etale cover} but adding $Y$ into it, we have the following commutative diagram:
$$
\xymatrix{
	A \ar@/^1.5pc/[rrrr]_{\mu_m} \ar[dd] & Y_m \ar[l]^{b_m} \ar[d]_{\pi_m} \ar@/^1.5pc/[dd]^{f'_m} \ar[rr]_{\nu_m} & & Y \ar[r]_{b} \ar[d]^{\pi} \ar@/_1.5pc/[dd]_{f'} & A \ar[dd] \\
	& X_m \ar[ul]^{a_m} \ar[d]_{f_m} \ar@/_1.5pc/[rr]_{\nu_m}& & X \ar[d]^f \ar[ur]_a & \\
	J(B) \ar@/_1.5pc/[rrrr]_{\mu_m} & B_m \ar[l] \ar[rr]_{\sigma_m} & & B \ar[r] & J(B)
}
$$
Here $\mu_m$ still denotes the multiplication-by-$m$ map of $A$ or $J(B)$, the Jacobian variety of $B$, $X_m$ and $f_m$ are just identical to those in \S \ref{subsection: etale cover}, $Y_m = Y \times_{\mu_m} A$, and 
$$
f'_m: Y_m \to B_m
$$
is the Stein factorization of the morphism $Y_m \to Y \to B$. Clearly, $X_m$ has at worst terminal singularities, and $\pi_m: Y_m \to X_m$ is also a resolution of singularities of $X_m$. Denote by $F'_m$ a general fiber of $f'_m$. Moreover, we will fix a sufficiently ample divisor $H$ on $A$. By \cite[Proposition 2.3.5]{Birkenhake_Lange},
\begin{equation} \label{eq: abelian multiplication}
	m^2H \equiv \mu_m^*H
\end{equation}

\subsection{Proof of Theorem \ref{main1}} 

We divide the proof into two cases.

\subsubsection{Case I: $g(B) > 0$} \label{subsubsection: g>0} We first prove Theorem \ref{main1} when $g:=g(B) > 0$. 

If $F$ is not of general type, neither is $F'$. In this case, for a general torsion element $\CP \in \Pic^0(Y)$, $f_* (\omega_{Y/B} \otimes \CP)$ is of rank $\chi(F', \omega_{F'}) = 0$. We deduce that $f_* (\omega_{Y/B} \otimes \CP) = 0$. By Proposition \ref{prop: chi_f degree}, $\chi_f = \chi_{f'} = 0$. Thus \eqref{main inequality} holds trivially.

From now on, we will always assume that $F$ is of general type. Set 
$$
L:=\pi^*K_{X/B}, \quad L_m:= \nu_m^*L = \pi_m^* K_{X_m/B_m}.
$$
Clearly, $L_m$ is nef, and $L_m|_{F'_m}$ is big. Since $X$ has at worst terminal singularities, $K_Y - \pi^*K_X$ is effective. Thus $K_{F'_m} - L_m|_{F'_m}$ is pseudo-effective. Moreover, since
$$
{f'_m}_* \CO_{Y_m}(\rounddown{L_m}) = {f_m}_* \omega_{X_m/B_m} =  {f'_m}_* \omega_{Y_m/B_m},
$$
by \cite[Main Theorem]{Fujita}, we deduce that ${f'_m}_* \CO_{Y_m}(\rounddown{L_m})$ is semi-positive.

Since $\deg \nu_m = m^{2q}$, we have
\begin{equation} \label{eq: growth1}
	L_m^n = m^{2q}L^n.
\end{equation}
There is a natural restriction morphism
$$
\nu_m|_{F'_m}: F'_m \to F'.
$$
It is an \'etale morphism and $\deg \nu_m|_{F'_m} = m^{2q-2g}$. Therefore, we deduce that
\begin{equation} \label{eq: growth2}
	(L_m|_{F'_m})^{n-1} = m^{2q-2g} (L|_{F'})^{n-1} = m^{2q-2g} K_F^{n-1}.
\end{equation}

Moreover, we claim that 
\begin{equation} \label{eq: growth3}
	\varepsilon(F'_m, L_m|_{F'_m}) \sim O(m^{2q-2g-2}).
\end{equation}
In fact, we may assume that $b^*H-L$ is pseudo-effective. By \eqref{eq: abelian multiplication}, $m^2(b_m^* H)-L_m$ is also pseudo-effective. Thus
\begin{align*}
	\varepsilon(F'_m, L_m|_{F'_m}, (b^*_m H)|_{F'_m}) & \le (m^2+1)^{n-2}\left((b_m^*H)|_{F'_m}\right)^{n-1} \\
	& \le \frac{2^{n-2}}{m^2} \left((b_m^*(\mu_m^*H))|_{F'_m}\right)^{n-1} \\
	& = \frac{2^{n-2}}{m^2} \left((\nu_m^*(b^*H)))|_{F'_m}\right)^{n-1}  \\
	& = 2^{n-2} m^{2q-2g-2} \left((b^*H)|_{F'}\right)^{n-1}.
\end{align*}
Thus the claim is verified. 

Now applying Proposition \ref{prop: slope} to $f'_m$ and $L_m$, we deduce that
\begin{equation} \label{eq: slope1}
	\left(1 + \frac{2n!(n-1) \varepsilon(F'_m, L_m|_{F'_m})} {(L_m|_{F'_m})^{n-1}} \right) L_m^n \ge 2n! \deg {f'_m}_* \CO_{Y_m}(\rounddown{L_m}).
\end{equation}
Recall that 
$$
{f'_m}_* \CO_{Y_m}(\rounddown{L_m}) = {f'_m}_* \omega_{Y_m/B_m}.
$$
Together with \eqref{eq: growth1}, \eqref{eq: growth2} and \eqref{eq: growth3}, the above inequality \eqref{eq: slope1} implies that
\begin{equation} \label{eq: slope2}
	\left(1 + O(m^{-2}) \right) K_{X/B}^n \ge 2n! \left(\frac{\deg {f'_m}_* \omega_{Y_m/B_m}}{m^{2q}} \right).
\end{equation}

Let $m \to \infty$. The left hand side of \eqref{eq: slope2} clearly tends to $K_{X/B}^n$. By Proposition \ref{prop: degree limit}, the right hand side tends to $\chi_{f'} = \chi(Y, \omega_Y) - \chi(F', \omega_{F'}) \chi(B, \omega_B)$, which is nothing but $\chi_f$. Thus the proof for $g>0$ is completed.

\subsubsection{Case II: $g(B) = 0$} \label{subsubsection: g=0} Now we prove Theorem \ref{main1} when $g(B) = 0$. It is easy to see that the argument for $g(B) > 0$ does not apply here directly. However, we can reduce this case to the previous one via a base change.

Choose four general distinct closed points $P_1$, \ldots, $P_4$ on $B$. Let $\sigma: C \to B$ be a double cover branched along $P_1$, \ldots, $P_4$. By the Hurwitz formula, $g(C) = 1$. Let $Y=X \times_B C$ and 
$$
f': Y \to C
$$
be the induced fibration. Thus we have the following commutative diagram:
$$
\xymatrix{
	Y \ar[d]_{f'} \ar[r]_{\pi} & X \ar[d]^f \\
	C \ar[r]_{\sigma} & B
}
$$
Since $f$ is relatively of maximal Albanese dimension, so is $f'$. As $g(C) = 1$, $Y$ itself is of maximal Albanese dimension. Since $P_1$, \ldots, $P_4$ are general, we deduce that $Y$ is normal. Moreover, we claim that $Y$ has at worst terminal singularities. In fact, let $\mu: X' \to X$ be a resolution of singularities of $X$. Then $X' \times_B  C \to Y$ is just a resolution of singularities of $Y$, and the claim is just an easy consequence of the adjunction.

Since $K_{Y/C} = \pi^*K_{X/B}$, $f'$ is also relatively minimal, and we have
\begin{equation} \label{eq: base change1}
	K_{Y/C}^n = 2K_{X/B}^n.
\end{equation}
We also have
$$
\pi_* \omega_Y = \omega_X \oplus \left(\omega_X \otimes \CO_X(2F)\right).
$$
from the above double cover. Thus from the adjunction formula, we deduce that
\begin{align} \label{eq: base change2}
	\chi_{f'} & = \chi(Y, \omega_Y) - \chi(C, \omega_C) \chi(F, \omega_F)  \nonumber \\
	& = \chi(X, \omega_X) + \chi(X, \omega_X \otimes \CO_X(2F)) \\
	& = 2 \chi(X, \omega_X) + 2 \chi(F, \omega_F) \nonumber  \\
	& = 2 \chi_f. \nonumber 
\end{align}
Now that $g(C) = 1 > 0$. We have 
$$
K_{Y/C}^n \ge 2n! \chi_{f'}
$$
as in \S \ref{subsubsection: g>0}. Together with \eqref{eq: base change1} and \eqref{eq: base change2}, it implies that
$$
K_{X/B}^n \ge 2n! \chi_f.
$$
Thus the whole proof of Theorem \ref{main1} is now completed.

\begin{remark} \label{rem: g>0}
	With this framework, it is easy to see that in order to get inequalities of the same type as \eqref{main inequality} with various slopes, we only need to (up to a base change to the $g(B) > 0$ case) replace \eqref{eq: slope1} by a corresponding explicit estimate with the same slope, and the same argument will give rise to the desired results. This is a crucial observation to us.
\end{remark}

\subsection{Sharper inequalities} As an example of the above remark, we can easily obtain the following result.

\begin{theorem} [Theorem \ref{main3}] \label{thm: sharper main} 
	Let $f: X \to B$ be a relatively minimal fibration from a variety $X$ of dimension $n \ge 3$ to a smooth curve $B$.  Denote by $F$ a general fiber of $f$. Suppose that $f$ is of maximal Albanese dimension and $a: X \to \Alb(X)$ is the Albanese map of $X$.
	\begin{itemize}
		\item [(1)] If $a|_F$ is birational, then 
		$$
		K_{X/B}^n \ge \frac{5n!}{2} \chi_f.
		$$
		
		\item [(2)] If $a|_F$ is not composed with an involution, then
		$$
		K_{X/B}^n \ge \frac{9n!}{4} \chi_f.
		$$
	\end{itemize}
\end{theorem}

\begin{proof}
	Remark \ref{rem: g>0} allows us to assume that $g > 0$. In the following, we just adopt the notation in \S \ref{subsection: notation}.
	
	To prove (1), note that now $B \to J(B)$ is an embedding. It implies that $a$ separates any two distinct fibers of $f$. In particular, $a$ is birational. Thus for every sufficiently large prime number $m > 0$, $b_m$ is birational. So is $b_m|_{F'_m}$. Then we simply replace the estimate \eqref{eq: slope1} in the proof of Theorem \ref{main1} by the inequality in Proposition \ref{prop: sharper slope} (1) for $f'_m$ and $L_m$, and the conclusion will follow by letting $m \to \infty$.
	
	The proof of (2) is similar. In this case, we know that $b|_{F'}$ is not composed with an involution. Let $d = \deg a|_F = \deg b|_{F'}$. By the following Lemma \ref{lem: involution}, $b_m|_{F'_m}$ is not composed with an involution as long as $m > d$. Thus the conclusion will follow similarly by letting $m \to \infty$.
\end{proof}

\begin{lemma} \label{lem: involution}
	Let $\alpha: V \to W$ be a generically finite morphism between two varieties of degree $d > 0$ such that $\alpha$ is not composed with an involution. Let $p > d$ be any prime number. Let $W_p \to W$ be a Galois cover with $G = \mathrm{Gal}(W_p/W)$ a $p$-group. Let $V_p: = V \times_W W_p$ and let $\alpha_p: V_p \to W_p$ be the induced morphism. Then $\alpha_p$ is not composed with an involution.
\end{lemma}

\begin{proof}
	By our assumption, $K(V) \simeq \frac{K(W)(t)}{(f(t))}$, where $f(t)$ is an irreducible polynomial of degree $d$ with coefficients in $K(W)$. Using Galois theory, we can find a variety $U$ and a generically finite map $\beta: U \to V$ such that $K(U)$ is the splitting field of $f(t)$. Thus $K(U)/K(W)$ is a Galois extension. Write
	$$
	H = \Gal(K(U)/K(W)).
	$$
	Then $H$ is a subgroup of $S_d$. In particular, $|H|$ divides $d!$. Since $p > d$ and $G$ is a $p$-group, we have $(|G|, |H|) = 1$.
	
	Let $U_p = U \times_V V_p$. We claim that $U_p$ is irreducible. Otherwise, let $U'_p$ be an irreducible component of $U_p$. Now the morphism $U'_p \to W$ has two factorizations $U'_p \to U \to W$ and $U'_p \to W_p \to W$. Thus both $|H|$ and $|G|$ divide $[K(U'_p) : K(W)]$. Since $(|G|, |H|) = 1$,  we have
	$$
	[K(U'_p) : K(W)] \ge |G| |H|.
	$$ 
	On the other hand, since the degree of the map $U'_p \to V_p$ is strictly less than $\deg \beta$, we have
	$$
	[K(U'_p) : K(W)] = [K(U'_p):K(W_p)][K(U):K(W)] < |G| |H|.
    $$
    This is a contradiction. As a result, $U_p$ is irreducible.
    In particular, the natural morphism $U_p \to U$ is also a Galois cover and 
	$$
	G = \Gal(K(U_p)/K(U)).
	$$

	We claim that the extension $K(U_p)/K(W)$ is also Galois. Write 
	$$
	G_p = \Aut(K(U_p)/K(W)).
	$$ 
	It is clear that 
	$$
	|G_p| \le [K(U_p):K(W)] = [K(U_p):K(U)][K(U):K(W)] = |G||H|.
	$$
	On the other hand, since $H = \Gal(K(U_p)/K(W_p))$, we may view both $G$ and $H$ as subgroups of $G_p$. Since $(|G|, |H|) = 1$, we deduce that 
	$$
	|G_p| \ge |G| |H|.
	$$
	Therefore, $|G_p| = |G| |H|$ and the claim is verified. As a consequence of this claim, $G$ is a normal subgroup in $G_p$.
	
	Now suppose that $\alpha_p$ is composed with an involution. This means that there exists a variety $V'_p$ such that $K(V_p) \supset K(V'_p) \supseteq K(W_p)$ and
	$$
	[K(V_p):K(V'_p)] = 2.
	$$
	Write $H_1=\Aut(K(U_p)/K(V_p))$ and $H'_1=\Aut(K(U_p)/K(V'_p))$. Then the fundamental theorem of Galois theory tells us that $H_1 \subset H'_1$ are both subgroups of $G_p$ and 
	$$
	[H'_1:H_1] = 2.
	$$
	Since $G$ is normal, we consider another two subgroups $H_1G \subset H'_1G$ of $G_p$. Then we still have
	$$
	[H'_1G :H_1G] = 2.
	$$
	Note that $K(U_p)^{H_1G} = K(V)$, by fundamental theorem of Galois theory again, $K(U_p)^{H'_1G}$ is a subfield of $K(V)$ and
	$$
	[K(V):K(U_p)^{H'_1G}] = 2.
	$$ 
	This implies that $\alpha$ is composed with an involution. However, this is absurd. Thus the proof is completed.
\end{proof}

\begin{remark}
	After we finished the first version of the paper, Barja informed us the result \cite[Lemma 2.9]{Barja_Pardini_Stoppino} which states that if one further assumes that $V$ is of general type, then $\Gal(\alpha_p) = \Gal(\alpha)$ for any prime number $p$ larger than a certain non-explicit constant depending on the volume and the dimension of $V$.
\end{remark}

\subsection{An example} We provide an example showing that \eqref{main inequality} is sharp.

Let $Y:=B \times A$ be a product of a smooth curve $B$ of genus $g$ and an abelian variety $A$ of dimension $n-1$, with two natural projections $p_1: Y \to B$ and $p_2: Y \to A$. Take two sufficiently ample divisors $L_1$ on $B$ and $L_2$ on $A$, respectively. Denote $L = p_1^*L_1 + p_2^*L_2$. Choose a smooth divisor $D \in |2L|$ on $Y$. Let $\pi: X \to Y$ be a double cover branched along $D$. It is easy to see that
$$
f: X \to B
$$
is a relatively minimal fibration whose general fiber $F$ is a double cover of $A$ branched along $L_2$, thus is of general type. Moreover, $f$ is relatively minimal of maximal Albanese dimension.

Since $K_{X/B} \sim \pi^*L$, we have 
$$
K_{X/B}^n = (\pi^*L)^n = 2L^n = 2n (\deg L_1)L_2^{n-1}
$$
On the other hand, since 
$$
\pi_* \omega_X = \CO_Y(p_1^*K_B) \oplus \CO_Y(L + p_1^*K_B)
$$
and 
$$
\pi_* \omega_F = \CO_A \oplus \CO_A(L_2),
$$
by the K\"unneth formula, we have
\begin{align*}
\chi(X, \omega_X) & = \chi(Y, \CO_Y(p_1^*K_B)) + \chi(Y, \CO_Y(L + p_1^*K_B)) \\ 
& = \chi(B, \omega_B) \chi(A, \CO_A) +  \chi(B, \CO_B(L_1+K_B)) \chi(A, \CO_A(L_2)) \\
& =  \chi(B, \CO_B(L_1+K_B)) \chi(A, \CO_A(L_2)),
\end{align*}
and 
$$
\chi(F, \omega_F) = \chi(A, \CO_A) + \chi(A, \CO_A(L_2)) = \chi(A, \CO_A(L_2)).
$$
It follows that
$$
\chi_f = \chi(A, \CO_A(L_2)) \left( \chi(B, \CO_B(L_1+K_B)) - \chi(B, \omega_B) \right) = \frac{L_2^{n-1}}{(n-1)!} \deg L_1.
$$
Thus for this fibration $f$, we have $K_{X/B}^n = 2n! \chi_f > 0$.

\subsection{Proof of Theorem \ref{main2}} \label{subsection: proof of main2}

Since the result is either known or trivial when $n = 2$, in the following, we assume that $n \ge 3$ and
$$
K_{X/B}^n = 2n! \chi_f > 0.
$$
We first prove Theorem \ref{main2} (1). Via a base change argument as in \S \ref{subsubsection: g=0}, we may assume that $g(B) > 0$. Thus we are under the setting of \S \ref{subsection: notation}. Moreover, by Theorem \ref{main3}, we know that $a|_F$ is composed with an involution.

Resume all notation in \S \ref{subsection: notation}. Write $\Sigma = a(F)$. Then $\Sigma$ is a subvariety of an abelian variety $A_F$, a general fiber of $A \to J(B)$ of dimension $q - g(B)$, and $\Sigma$ generates $A_F$. To show that $\Sigma = A_F$, we only need to show that the smooth model of $\Sigma$ has Kodaira dimension zero.

Let $\sigma: \Sigma' \to \Sigma$ be a resolution of singularities of $\Sigma$. Let $\Sigma_m = a_m(F_m)$ and $\Sigma'_m = \Sigma_m \times_{\sigma} \Sigma'$. Then $\sigma_m: \Sigma'_m \to \Sigma_m$ is also a resolution of singularities of $\Sigma_m$. Let $\nu_m: \Sigma'_m \to \Sigma'$ be the induced \'etale map. Thus we have the following diagram:
$$
\xymatrix{
	& \Sigma'_m \ar[r]_{\nu_m} \ar[d]^{\sigma_m} & \Sigma' \ar[d]_{\sigma} &  \\
	F'_m \ar[r] \ar[dr]^{b_m} & \Sigma_m \ar@{^{(}->}[d] \ar[r] & \Sigma \ar@{^{(}->}[d] & F' \ar[dl]_{b} \ar[l] \\
	& A_F \ar[r]^{\mu_m} & A_F &
}
$$
Denote 
$$
r' = r'(L|_{F'}, (b^*H)|_{F'}, \Sigma'), \quad r'_m = r'(L_m|_{F'_m}, (b^*_m H)|_{F'_m}, \Sigma'_m).
$$ 
With this notation, by \eqref{eq: abelian multiplication}, we have
$$
r':= \frac{(L|_{F'}) \left((b^*H)|_{F'} \right)^{n-2}}{K_{\Sigma'} \left( \sigma^*(H|_\Sigma) \right)^{n-2}} = \frac{(L_m|_{F'_m}) \left((b_m^*H)|_{F'_m} \right)^{n-2}}{K_{\Sigma'_m} \left(\sigma_m^*(H|_{\Sigma_m})\right)^{n-2}} =: r'_m.
$$
It simply implies that
$$
\delta(r') = \delta(r'_m).
$$
Now we use the framework of the proof of Theorem \ref{main1} again and replace \eqref{eq: slope1} by the one in Proposition \ref{prop: sharper slope} (3). Together with the above equality, we deduce that 
$$
K_{X/B}^n \ge \frac{(5\delta(r') - 3) n!}{2 \delta(r') - 1} \chi_f.
$$
However, if $\kappa(\Sigma') > 0$, we would have $\delta(r') > 1$ and thus $\frac{5\delta(r') - 3}{2 \delta(r') - 1} > 2$. This is a contradiction. As a result, $\kappa(\Sigma') = 0$ and $\Sigma = A_F$.

Now we prove Theorem \ref{main2} (2). Note that $K_{X/B}^n > 0$ implies that $K_{X/B}$ is also big. In particular, a general fiber $F$ of $f$ is a minimal variety of general type. By \cite[Theorem 1-2-5]{Kawamata_Matsuda_Matsuki}, we have 
$$
R^if_*\omega_{X/B}^{[l]} = R^if_*\omega_X^{[l]} \otimes \omega_B^{\otimes (-l)} =0
$$ 
for any $i>0$ and $l \ge 2$. Thus for any $l \ge 2$, we have 
$$
\chi(B, f_*\omega_{X/B}^{[l]})=\chi(X, \omega_{X/B}^{[l]}).
$$
Let $P_l(F)$ denote the $l^{\mathrm{th}}$ plurigenus of $F$. Then we have
\begin{align*}
    \deg f_*\omega_{X/B}^{[l]} & = \chi(B, f_*\omega_{X/B}^{[l]}) - P_l(F) \chi(B, \CO_B) \\
    & = \chi(X, \omega_{X/B}^{[l]}) - P_l(F) \chi(B, \CO_B) \\
    & = \frac{l^n K_{X/B}^n}{n!} + o(l^n).
\end{align*}
In particular, for $l \gg 0$, $\det f_*\omega_{X/B}^{[l]}$ is an ample line bundle on $B$. By \cite[Proposition 4.6]{Viehweg},  we know that for $l \gg 0$, the vector bundle $f_*\omega_{X/B}^{[l]}$ is ample. Thus by \cite[Theorem 1.4]{Ohno}, $mK_{X/B} - F$ is nef for a sufficiently large $m \in \ZZ$. Replacing $B$ by one of its cyclic cover of degree $m$ which is either \'etale (if $g(B) > 0$) or ramified at general points (if $g(B) = 0$) and replacing $f: X \to B$ by the fibration induced by this base change accordingly,  we may assume that $K_{X/B} - F$ is nef. Similar to \S \ref{subsubsection: g=0}, we know that this induced fibration is also relatively minimal and of maximal Albanese dimension. Moreover, we still have 
$$
K_{X/B}^n = 2n! \chi_f > 0
$$ 
for this new fibration $f$.

Using the same strategy as in the proof of Theorem \ref{main1} but replacing $K_{X/B}$ by $K_{X/B}-F$, we deduce that
$$
(K_{X/B}-F)^n\ge 2n!\deg f_*(\omega_{X/B} \otimes \CP \otimes \CO_X (-F)),
$$
where $\CP\in \Pic^0(X)$ is a general torsion element. That is,
$$
K_{X/B}^n - nK_F^{n-1} \ge 2n!(\chi_f - \chi(F, \omega_F)).
$$
By the assumption that $K_{X/B}^n = 2n! \chi_f$, we have
$$
K_F^{n-1} \le 2(n-1)! \chi(F, \omega_F).
$$
Since $F$ is minimal of maximal Albanese dimension, together with the absolute Severi inequality for $F$, we deduce that 
$$
K_F^{n-1}=2(n-1)!\chi(F, \omega_F).
$$
Thus the proof is completed.

\bibliography{relative_severi}
\bibliographystyle{amsplain}
	
\end{document}